\documentclass[11pt]{article}
\usepackage{amsmath,amsthm,amssymb,amsfonts}
\usepackage[width=16.5cm,height=23cm]{geometry}
\usepackage{enumerate}
\usepackage{mathrsfs}
\usepackage[cmtip,all]{xy}
\usepackage{tikz}

\usepackage{ifxetex}
\input ifxetex.sty
\usepackage{enumerate}
\usepackage{mathrsfs}
\usepackage[cmtip,all]{xy}
\usepackage{tikz}
\usetikzlibrary{shapes.geometric}
\usepackage[final]{hyperref}
\definecolor{sq3sq1color}{rgb}{0.5,0,0.5}
\definecolor{sq2color}{rgb}{0.1,0.7,0.1}
\colorlet{taucolor}{red}
\colorlet{partialsq2sq1color}{sq3sq1color!53!black}
\colorlet{incsq2sq1color}{sq3sq1color!67!green}
\colorlet{sq2rhosq1color}{taucolor!45!sq2color}
\colorlet{sq2prcolor}{sq2color!53!black}
\colorlet{incsq2color}{sq2color!67!yellow}
\colorlet{sq2partialcolor}{sq2color!42!blue}
\colorlet{tauprcolor}{taucolor!53!black}
\colorlet{taupartialcolor}{taucolor!42!yellow}

\usepackage{cleveref}
\RequirePackage[color       = blue!20,
                bordercolor = black,
                textsize    = tiny,
                figwidth    = 0.99\linewidth]{todonotes}

\theoremstyle{plain}
\newtheorem{theorem}{Theorem}[section]
\newtheorem*{theorem*}{Theorem}
\newtheorem{lemma}[theorem]{Lemma}
\newtheorem{proposition}[theorem]{Proposition}
\newtheorem*{proposition*}{Proposition}
\newtheorem{corollary}[theorem]{Corollary}
\newtheorem*{corollary*}{Corollary}
\theoremstyle{remark}
\newtheorem{remark}[theorem]{Remark}
\newtheorem{remark*}{Remark}
\theoremstyle{definition}
\newtheorem{defn}[theorem]{Definition}
\newtheorem{defn*}{Definition}

\newtheorem*{induction*}{Induction Hypothesis}

\newcommand{\im}{\mathrm{im}}
\newcommand{\supp}{\mathrm{supp}}
\newcommand{\RRR}{\mathscr{R}}
\newcommand{\XXX}{\mathscr{X}}

\newcommand{\QQQQ}{\mathscr{Q}}

\newcommand{\MHH}{\mathbf{MHH}}
\newcommand{\THH}{\mathbf{THH}}
\newcommand{\op}{\text{op}}

\newcommand{\MMM}{\mathscr{M}}
\newcommand{\NNN}{\mathscr{N}}
\newcommand{\QQQ}{\mathscr{S}}
\newcommand{\cc}{\mathsf{e}}
\newcommand{\AAA}{\mathcal{A}}

\newcommand{\GAMMA}{\Gamma}
\newcommand{\LAMBDA}{\Lambda}
\newcommand{\SSS}{\text{S}}

\newcommand{\TOR}{\mathbf{Tor}}
\newcommand{\PPP}{\text{P}}
\newcommand{\x}{\mathsf{x}}

\newcommand{\CC}{\mathbb{C}}

\newcommand{\FF}{\mathbb{F}}

\newcommand{\MM}{\mathbb{M}}
\newcommand{\NN}{\mathbb{N}}

\newcommand{\ZZ}{\mathbb{Z}}

\newcommand{\Sm}{\operatorname{Sm}}
\newcommand{\Spec}{\operatorname{Spec}}

\newcommand{\SH}{\mathscr{SH}}

\newcommand{\M}{\mathbf{M}}
\newcommand{\HHH}{\mathbf{H}}

\newcommand{\smsh}{\wedge}

\newcommand{\Tor}{\operatorname{Tor}}

\newcommand{\blambda}{\bar{\lambda}}
\newcommand{\btau}{\bar{\tau}}
\newcommand{\bxi}{\bar{\xi}}
\newcommand{\bmu}{\bar{\mu}}

\begin{document}

\title{Hochschild homology of mod-$p$ motivic cohomology over algebraically closed fields}
\author{Bj{\o}rn Ian Dundas, Michael A.~Hill, Kyle Ormsby, Paul Arne {\O}stv{\ae}r}
\date{\today}
\maketitle

\begin{abstract}
We perform Hochschild homology calculations in the algebro-geometric setting of motives.
The motivic Hochschild homology coefficient ring contains torsion classes which arise from the mod-$p$ 
motivic Steenrod algebra and from generating functions on the natural numbers with finite non-empty support.
Under the Betti realization, 
we recover B{\"o}kstedt's calculation of the topological Hochschild homology of finite prime fields. 
\end{abstract}

\section{Introduction}
\label{section:introduction}

Let $\RRR$ be a motivic ring spectrum such as algebraic cobordism, homotopy algebraic $K$-theory, 
or motivic cohomology \cite{MR1648048}.
Working in the stable motivic homotopy category $\SH(F)$ of a field $F$, 
we define the motivic Hochschild homology $\MHH(\RRR)$ of $\RRR$ as the derived tensor product 
\begin{equation}
\label{equation:mhh0}
\RRR
\wedge_{\RRR\wedge\RRR^{\op}}
\RRR.
\end{equation}
The concepts of Hochschild homology for associative algebras and topological Hochschild homology for 
structured ring spectra inspire our constructions.
In the event $\RRR$ is commutative one may equivalently to \eqref{equation:mhh0} form the tensor product
in the category of commutative motivic ring spectra
with 
the simplicial circle
\begin{equation}
\label{equation:mhh01}
S^{1}\otimes\RRR.
\end{equation}

The primary purpose of this paper is to calculate the homotopy groups of motivic Hochschild homology 
of $\M\FF_{p}$ over algebraically closed fields
--- the Suslin-Voevodsky mod-$p$ motivic cohomology ring spectrum for $p$ any prime number.
When the base field $F$ admits an embedding into the complex numbers $\mathbb{C}$, 
the Betti realization functor allows us to compare our $\MHH$ calculations with B{\"o}kstedt's pioneering 
work on topological Hochschild homology of the corresponding topological Eilenberg-Mac Lane spectrum $\HHH\FF_p$.
In fact, 
our calculation for $\MHH$ specializes to the one for $\THH$ in \cite{MB:THHZF}.
Additively, 
$\THH(\FF_p)$ splits as a restricted product of Eilenberg-MacLane spectra in the stable homotopy category.
This is not the case, however, for $\MHH(\FF_p)$, $\M\FF_{p}$, and $\SH(F)$.
The source of this extra layer of complexity is the abundance of $\tau$-torsion elements in the coefficients.
Here $\tau$ is a canonical class in the mod-$p$ motivic cohomology of $F$, which maps to the unit element in 
singular cohomology under Betti realization.
\vspace{0.05in}

We express the coefficient ring $\MHH_{\star}(\FF_p)$ in terms of algebra generators $\tau$, 
$\mu_{i}$, $\x_{S,f}$ arising from the mod-$p$ motivic Steenrod algebra \cite{MR3730515}, \cite{MR2031198}, 
and generating endofunctions $f\colon \NN\circlearrowleft$ with finite non-empty support containing some 
subset $S\subset\NN$.
The infinity of $\tau$-torsion classes $\x_{S,f}$ is not witnessed in $\THH_{\star}(\FF_p)$.
For example, 
Kronecker delta functions give rise to such classes 
(in this case, $S$ is either empty or a singleton set).

\begin{theorem}
\label{theorem:mainintroduction}
Over an algebraically closed field of exponential characteristic $\cc(F)\neq p$, 
there is an algebra isomorphism
\begin{equation}
\label{equation:maincalculation}
\MHH_{\star}(\FF_{p})
\cong 
\FF_{p}[\tau,\mu_{i},\x_{S,f}]_{i\in\NN,\, (S\subset \supp\, f,f\colon \NN\circlearrowleft)}/\mathcal{I}
\end{equation}
with the ideal of relations 
\[
\mathcal{I}=
\left(
  \begin{matrix}
    \mu_{i}^{p}-\tau^{p-1}\mu_{i+1},\\
\tau^{p-1}\x_{S,f},\\
\x_{S,f}\cdot\x_{T,g}-\underset{u\in\supp(f+g)-S\cup T}\sum\epsilon_u\cdot\x_{S\cup T\cup\{u\},f+g}
  \end{matrix}
\right).
\]
Here the support of $f$ is a finite non-empty subset of the natural numbers and $S\subset \supp f\subset \NN$ 
does not contain the minimal element of $\supp f$.  The coefficient $\epsilon_u\in\FF_p$ is given explicitly in \Cref{def:somenumbers}.
The algebra generators have bidegrees given by 
$|\tau|=(0,-1)$,
$|\mu_{i}|=(2p^{i},p^{i}-1)$,
and
\[
|\x_{S,f}|=
(|S|+1)(-1,p-1)+p\sum_{j\in\supp f}f(j)(2p^{j},p^{j}-1).
\]
\end{theorem}

Since the homotopy of \(\MHH(\FF_p)\) is not a free module over the homotopy of \(\M\FF_p\), we deduce a non-splitting of the motivic Hochschild homology in \(\M\FF_p\)-modules.

\begin{corollary}
    The motivic Hochschild homology of \(\FF_p\) does not split as a wedge of suspensions of \(\M\FF_{p}\).
\end{corollary}

This gives a surprising obstruction to classical results about topological Hochschild homology and Thom spectra. Mahowald showed that the Eilenberg--MacLane spectrum \(\HHH\FF_2\) is a Thom spectrum of a double loop map with source \(\Omega^2 S^3\) \cite{MR445498}. Behrens--Wilson showed that an analogous result is true \(C_2\)-equivariantly, with the base now \(\Omega^{2,1}S^{3,1}\) \cite{MR3856165}. Blumberg--Cohen--Schlichtkrull showed that the topological Hochschild homology of Thom spectra are Thom spectra, and when the topological \(\eta\) vanishes, these split as smash products of the original Thom spectrum and a space related to the classifying space of the base \cite{MR2651551}. Equivariantly, classically and \(C_2\)-equivariantly, this splits as a wedge of smash powers of spheres. Putting this all together, we cannot have all of these results hold in the motivic setting.

As a guide to this paper, 
we outline the proof of \Cref{theorem:mainintroduction} and explain how the algebra generators arise in our context.
The key idea in proving our results is to study the $\tau$-inversion and mod-$\tau^{p-1}$ reduction
of $\MHH_{\star}(\FF_{p})$.
We review some background and set our notation in \Cref{section:MHH}.
In \Cref{section:acf} we divide the proof of \Cref{theorem:mainintroduction} into the following steps.
\begin{enumerate}
\item[Step 1] 
\Cref{theorem:tauinvertedMHH} calculates the $\tau$-inverted or \'etale motivic Hochschild homology 
\begin{equation}
\label{equation:mhhtauinverted}
\MHH_{\star}(\FF_{p})[\tau^{-1}]
\cong
\FF_{p}[\tau^{\pm 1},\mu_{i}]_{i\geq 0}/(\mu_{i}^{p}-\tau^{p-1}\mu_{i+1})
\cong
\FF_{p}[\mu,\tau^{\pm 1}]
\cong
\THH_{\ast}(\FF_{p})[\tau^{\pm 1}].
\end{equation}
The generator $\mu$ in \eqref{equation:mhhtauinverted} has bidegree $(2,0)$.
Hence all the classes $\mu_{i}$,
$i\in\NN$, 
generate the non-$\tau$-torsion part of $\MHH_{\star}(\FF_{p})$ subject to the relation $\mu_{i}^{p}=\tau^{p-1}\mu_{i+1}$.
\item[Step 2] 
\Cref{theorem:MHHMFCtau} calculates the coefficients of mod-$\tau^{p-1}$ motivic Hochschild homology
\begin{equation}
\label{equation:mhhtaureduction}
\MHH_{\star}(\FF_p)/\tau^{p-1}
\cong
(\bigotimes_{i\geq 0}
\GAMMA_{\FF_{p}}(\bmu_{i})
\otimes 
\LAMBDA_{\FF_{p}}(\blambda_{i+1}))
\otimes
\FF_{p,\tau}.
\end{equation}
The bidegrees of the generators are $|\blambda_{i}|=(2p^{i}-1,p^{i}-1)$, $|\bmu_{i}|=(2p^{i},p^{i}-1)$, 
and $\FF_{p,\tau}$ is shorthand for $\FF_{p}[\tau]/\tau^{p-1}$.
The divided powers algebra generator $\bmu_{i}$ is the image of $\mu_{i}\in\MHH_{\star}(\FF_p)$.
We note that \eqref{equation:mhhtaureduction} coincides with the $E^{2}$ page of the 
Tor spectral sequence for $\MHH(\FF_p)/\tau^{p-1}$.
In fact, 
the said Tor spectral sequence collapses at $E^{2}$ with no multiplicative extensions.
\item[Step 3] 
\Cref{lemma:nontrivialB} shows the $\tau^{p-1}$-Bockstein of $\gamma_{j}\bar\mu_{i}$ equals 
$\bar\lambda_{i+1}\gamma_{j-p}\bar\mu_{i}$.
First we establish the case $j=p$, 
and the rest follows by shuffle products in the bar construction of $\MHH_{\star}(\FF_p)/\tau^{p-1}$.
Here, 
the $\tau^{p-1}$-Bockstein on $\MHH_{\star}(\FF_{p})$ is the composite of the canonical boundary and quotient maps in
\begin{equation}
\label{equation:Bocksteindefinition}
\bar \partial
\colon
\MHH_{\ast+1,\ast}(\FF_{p})/\tau^{p-1}
\xrightarrow{\partial}
\MHH_{\ast,\ast+p-1}(\FF_{p})
\xrightarrow{q}
\MHH_{\ast,\ast+p-1}(\FF_{p})/\tau^{p-1}.
\end{equation}
In \Cref{corollary:Bhomology}, 
we conclude the Bockstein homology of $\MHH_{\star}(\FF_{p})/\tau^{p-1}$ is isomorphic to the graded commutative 
$\FF_{p,\tau}$-algebra $\bigoplus_{i\geq 0}\LAMBDA_{\FF_{p,\tau}}(\bmu_{i})$.
\item[Step 4] 
\Cref{lem:Einftyanal1} shows the $\tau$-torsion classes in $\MHH_{\star}(\FF_{p})$ inject into 
$\MHH_{\star}(\FF_{p})/\tau^{p-1}$ with image that of the $\tau^{p-1}$-Bockstein $\bar\partial$ 
(degrees are made explicit through generating functions).
Moreover, 
the reduction map $q$ sends the image of the boundary $\partial$ isomorphically to the image of the 
Bockstein $\bar\partial$.
\item[Step 5] 
If $f\colon \NN\circlearrowleft$ has finite support and $S\subseteq\supp f$, 
we set
$$
\chi_{S,f}
=
\left(\prod_{m\in S}\blambda_{m+1}\gamma_{pf(m)-p}\bmu_m\right)
\left(\prod_{n\not\in S}\gamma_{pf(n)}\bmu_n\right)
\in
\MHH_{\star}(\FF_{p})/\tau^{p-1}.
$$
We define the $\tau$-torsion algebra generators in \Cref{theorem:mainintroduction} by
$$
\x_{S,f}
=
\partial\chi_{S,f}
\in
\MHH_{\star}(\FF_{p}).
$$
In particular, 
$\chi_{\emptyset,0}=1$, 
$\chi_{\emptyset,p^j\delta_n}=\gamma_{p^{j+1}}\bmu_n$ and $\chi_{\{m\},\delta_m}=\lambda_{m+1}$.
Here $\delta_n: \NN\circlearrowleft$ is zero except for $\delta_n(n)=1$.  
The Bockstein yields $\x_{S,f}=\sum_{n\in\supp(f)-S}\chi_{S\cup\{n\},f}$ since 
$\bar\partial\gamma_n\bmu_{i}=\blambda_{i+1}\gamma_{n-p}\bmu_{i}$, $\bar\partial\blambda_{i}=0$, 
and $\bar\partial$ is a derivation. Since the classes $\bmu_{i}$, 
$\chi_{S,f}$, and the $\bar\partial$ cycles $\lambda_{i+1}=\bar\partial\gamma_{p}\bmu_{i}$ 
generate $\MHH_{\star}(\FF_{p})/\tau^{p-1}$, 
the classes $\bmu_{i}$ and $\x_{S,f}$ generate the boundary.
\item[Step 6] 
By combining the $\tau$-inverted and mod-$\tau^{p-1}$ calculations we finally deduce \eqref{equation:maincalculation}.
The relation $\mu_{i}^{p}=\tau^{p-1}\mu_{i+1}$ is rooted in the mod-$p$ motivic Steenrod algebra.
The Bockstein calculation shows the vanishing $\tau^{p-1}\x_{S,f}=0$. 
\Cref{cor:HDandZD} shows the multiplicative relation between the $\x_{S,f}$ classes 
follows from a similar formula for the $\chi_{S,f}$ classes.
We refer to \Cref{def:somenumbers} for the entity $\epsilon_{u}$.

For example, 
at the prime $p=2$, 
we obtain the relations
$$
\x_{\delta_0+\delta_1}\x_{\delta_2}+\x_{\delta_1+\delta_2}\x_{\delta_0}+\x_{\delta_2+\delta_0}\x_{\delta_1}
=
0,
$$
$$
\x_{\delta_0+\delta_1}\x_{\delta_1+\delta_2}
=
\x_{2\delta_1}\x_{\delta_0+\delta_2}.
$$
\end{enumerate}

\Cref{theorem:mainintroduction} admits a succinct reformulation in terms of naturally induced pullback squares
of commutative $\FF_p[\tau]$-algebras given in \Cref{sec:even} and  \Cref{sec:multex}.
For example, 
when $p=2$, 
we note the pullback square of commutative $\FF_p[\tau]$-algebras 
$$
\xymatrix{
   \MHH_{\star}(\FF_2)\ar[r]\ar[d]&\FF_p[\tau,\mu_i]/(\mu_i^2-\tau\mu_{i+1})\ar[d]\\
   \FF_2[\tau,\bmu_{i},\x_{S,f}]/\mathcal{I}
   \ar[r]&\FF_2[\tau,\bmu_i]/(\bmu_i^2,\tau)}
$$
where the ideal of relations is given by 
$$
\mathcal{I}
=
\left(\tau,\bmu_{i}^2,\x_{S,f}\cdot\x_{T,g}-\sum_{t_{f+g}\not=u\in\supp(f+g)-S\cup T}
\epsilon_{u}\cdot\x_{S\cup T\cup\{u\},f+g}\right).
$$
Our calculation shows the left vertical map in the pullback is an isomorphism on $\tau$-torsion classes.
Furthermore,
the upper horizontal map is an injection on non-$\tau$-torsion classes.
An analogous result holds for all odd primes.

\subsection{Notation}
This paper uses the following notation.
\vspace{0.05in}

\begin{tabular}{l|l}
$p$, $F$ & prime number, base field of exponential characteristic $\cc(F)\neq p$ \\
$\Sm_{F}$ & smooth separated schemes of finite type over $\Spec(F)$ \\
$\SH(F)$ & stable motivic homotopy category of $F$ \\ 
$\RRR$ & motivic ring spectrum \\
$H^{\ast,\ast}$, $h^{\ast,\ast}$ & integral, mod-$p$ motivic cohomology groups of $F$ \\
$K^{M}_{\ast}$, $k^{M}_{\ast}$ & integral, mod-$p$ Milnor $K$-groups of $F$ \\
$\MM_{\star}$ & mod-$p$ motivic homology ring of $F$ \\
$\AAA_{\star}$ & dual motivic Steenrod algebra of $F$ at $p$ \\
$\FF_{p,\tau}$ & shorthand for $\FF_{p}[\tau]/\tau^{p-1}$ \\
$\GAMMA$, $\LAMBDA$, $\SSS$ & divided power, exterior, symmetric algebras
\end{tabular}

\section{Motivic Hochschild homology}
\label{section:MHH}

\begin{defn}
\label{defn:mhh}
Let $\RRR$ be a motivic ring spectrum.  The \emph{motivic Hochschild homology} of an $\RRR$-bimodule $\MMM$ is the derived smash product 
\begin{equation}
\label{equation:mhh1}
\MHH(\RRR;\MMM)
:=
\MMM
\wedge_{\RRR\wedge\RRR^{\op}}
\RRR
\end{equation}
in $\SH(F)$.
\end{defn}

When $\RRR=\MMM$, 
the derived tensor product \eqref{equation:mhh1} specializes to $\MHH(\RRR)$ in \eqref{equation:mhh0} 
or equivalently \eqref{equation:mhh01} in the event $\RRR$ is commutative.
If $\RRR\rightarrow\QQQQ$ is a map of motivic ring spectra and $\MMM$ is a $\QQQQ$-$\RRR$ bimodule, 
then reassociating the smash factors implies the equivalence
\begin{equation}
\label{equation:mhh2}
\MHH(\RRR;\MMM)
\simeq
\MMM
\wedge_{\QQQQ\wedge\RRR^{\op}}
\QQQQ.
\end{equation}
In the following, 
we assume that $\RRR$ is a cofibrant commutative motivic ring spectrum in any of the model categorical approaches 
to $\SH(F)$ as in \cite{zbMATH02028920}, \cite{zbMATH01698557}, \cite{zbMATH01887476}, \cite{zbMATH01527083}
(this assumption is superfluous in the $\infty$-category of motivic spectra \cite{zbMATH06374152}).
Commutative motivic ring spectra are cotensored over motivic spaces by the usual adjunctions.
We will only need the special case of simplicial sets or topological spaces.  
The case of finite simplicial sets is particularly transparent since it derives from the relation 
$\{1,\dots,n\}\otimes\RRR=\RRR^{\smsh n}$.
The assignment $\XXX\mapsto\XXX\otimes\RRR$ from motivic spaces to motivic ring spectra has several 
useful properties which generalize from the topological setting and which we will use freely.
\begin{itemize}
\item $\XXX\mapsto\XXX\otimes\RRR$ is homotopy invariant and preserves coproducts 
(and so in particular sends pushouts to smashes).
\item $*\otimes\RRR\cong\RRR$, $S^0\otimes\RRR\cong\RRR\smsh\RRR$ and 
(since $S^1$ is the derived pushout of $*\gets S^0\to*$)  
$\MHH(\RRR;\MMM)\simeq \MMM\smsh_{\RRR}(S^1\otimes\RRR)$.
\item The product on $\XXX\otimes\RRR$ is induced by the fold $\XXX\coprod\XXX\to \XXX$.
\item Choosing a point $*\to\XXX$ makes $\XXX\otimes\RRR$ an augmented commutative $\RRR$-algebra.
\item The inclusion $\{-1,1\}\subseteq\{-1,0,1\}\cong \{0,-1\}\vee\{0,1\}$ induces the comultiplication 
$\RRR\smsh\RRR\to \RRR\smsh\RRR\smsh\RRR\cong(\RRR\smsh\RRR)\smsh_\RRR(\RRR\smsh\RRR)$ and the nontrivial map 
$\{-1,1\}\to\{-1,1\}$ gives the anti-involution of the ``dual Steenrod $\RRR$-Hopf algebroid'' 
$S^0\otimes\RRR=\RRR\smsh\RRR$ (alge\emph{broid} since the maps involved are not pointed, 
and so there is no guarantee that the units corresponding to the two choices of base points will coincide).
The suspension of these maps give the pinch map
$$
S^1\cong[-1,1]\coprod_{\{-1,1\}}*\to [-1,1]\coprod_{\{-1,0,1\}}*\cong S^1\vee S^1
$$ 
and the flip map $S^1\to S^1$ both of which are pointed maps, 
inducing the $\RRR$-Hopf algebra structure
$$
\psi\colon\xymatrix{S^1\otimes\RRR\ar[r]&(S^1\vee S^1)\otimes\RRR\cong 
(S^1\otimes\RRR)\smsh_{\RRR}(S^1\otimes\RRR)}, \qquad \chi\colon S^1\otimes\RRR\cong S^1\otimes\RRR
$$ 
on the ``motivic Hochschild homology''
--- 
to implement this using finite simplicial models of the circle, one must subdivide.

Hence, 
if $\MHH_{\star}(\RRR)$ is flat over $\RRR_{\star}$, 
which will turn out not to be true for $\RRR=\M\FF_p$, 
we get an $\RRR_\star$-Hopf algebra structure on $\MHH_{\star}(\RRR)$.
\item The tensor with spaces in the category of motivic spectra is $X{}\mapsto X{}_+\smsh\RRR$ and 
the universal property defines a unique map of motivic spectra
$$
\sigma^+\colon X{}_+\smsh\RRR\to X{}\otimes\RRR.
$$
If $X$ is a set considered as a motivic space, 
the inclusion of the points $\{x\}\subseteq X$ induces the desired map 
$X_+\smsh\RRR\cong\bigvee_{\{x\}\in X}\{x\}\otimes\RRR\to X\otimes\RRR$.
If $X{}$ is already pointed, 
the basepoint in $X{}$ makes $X{}\otimes\RRR$ an $\RRR$-algebra, 
giving rise to the  free extension to an $\RRR$-linear map
$$
\sigma\colon\xymatrix{\RRR\smsh X_+\smsh\RRR\ar[r]^-{1\otimes\sigma^+}&\RRR\smsh 
X{}\otimes\RRR\ar[r]^-{\text{mult.}}&X{}\otimes\RRR}.
$$
\item In general, 
if $\AAA$ is a commutative $\RRR$-algebra then $X\mapsto\AAA^X=\hom_\RRR(X_+\smsh\RRR,\AAA)$ 
is a cotensor (doesn't depend on $\RRR$).    
The unit of adjunction 
$$
\alpha^\RRR\colon \AAA\to(X\otimes^\RRR\AAA)^X
$$
is a map of commutative $\RRR$-algebras.
In the category of $\RRR$-modules, 
the adjoint of $\alpha^\RRR$ takes the form
$$
\xymatrix{
\sigma^\RRR\colon(X_+\smsh\RRR)\smsh_\RRR\AAA\ar[r]^-{1\smsh\alpha^\RRR}
&(X_+\smsh\RRR)\smsh_\RRR(X\otimes^\RRR\AAA)^X\ar[r]^-{\text{ev}}&(X\otimes^\RRR\AAA),
}
$$
where $\text{ev}$ is the evaluation.
  
Composition gives an $\RRR$-algebra map $\nu\colon \RRR^X\smsh_\RRR(X\otimes^\RRR\AAA)\to (X\otimes^\RRR\AAA)^X$. 
Assume that $X$ is a finite cell complex and that $\pi_\star(X_+\smsh \RRR)$ is a finitely generated 
$\pi_\star\RRR$-module with basis $\mathcal B$. 
Then $\nu$ is an equivalence and
$$
\nu^{-1}_*\alpha_\star\colon\pi_\star\AAA
\to
\hom_{\pi_\star\RRR}(\pi_\star(X_+\smsh\RRR),\pi_\star\RRR)\otimes_{\pi_\star\RRR}\pi_\star(X\otimes\AAA)
$$ 
satisfies
$$
\nu^{-1}_\star\alpha_\star(a)=\sum_{x\in\mathcal B}x^*\otimes\sigma^\RRR_*(x\otimes a).
$$ 
Here $x^*$ is the basis element dual to $x$ and 
$x\otimes a\in\pi_\star(X_+\smsh\RRR)\otimes_{\pi_\star\RRR}\pi_\star\AAA=\pi_\star((X_+\smsh\RRR)\smsh_\RRR\AAA)$.
We'll use this formula in \Cref{lem:extension} to get a relation in $\MHH(\FF_p)$
(in the topological case,
see \cite[\S5]{zbMATH02221879} for $X=S^1$ using the circle action).

Note that $X\otimes^\RRR(\RRR\smsh\RRR)\cong\RRR\smsh(X\otimes\RRR)$ is the tensor of $X$ with $\RRR\smsh\RRR$ 
in the category of commutative $\RRR$-algebras,  
and there is a commutative diagram
$$\xymatrix{
(X_+\smsh\RRR)\smsh_\RRR(\RRR\smsh\RRR)\ar[r]^-{\sigma^\RRR}\ar@{=}[d]^\cong&X\otimes^\RRR(\RRR\smsh\RRR)
\ar@{=}[d]^\cong\\
\RRR\smsh X_+\smsh\RRR\ar[r]^-{1\smsh\sigma^+}&\RRR\smsh X\otimes\RRR}
$$
where the vertical isomorphisms are the associators.
\end{itemize}

\subsection{Comparison of simplicial models}
\label{sec:compsimp}

It will be convenient to make explicit some of the simplicial models and how they interact.  In this subsection, let $I=\Delta[1]$ be the simplicial interval with boundary $S^0=\partial\Delta[1]$ and let $S^1=I\coprod_{S^0}*$ be the simplicial circle.  The subdivision of the circle relevant for the comultiplication is
$dS^1=(I\coprod I)\coprod_{S^0\coprod S^0}S^0$ with weak equivalence $dS^1\to S^1$ given by sending the first interval to the base point.  The pinch map $\nabla\colon dS^1\to S^1\vee S^1$ identifies the endpoints.  It is sometimes convenient to write $dS^1$ as $*\coprod_{S^0}(I\times S^0)\coprod_{S^0}*$.
Under the canonical isomorphism $\RRR=*\otimes\RRR$ we get an identification $S^1\otimes\RRR=(I\otimes^\RRR\RRR)\smsh_{S^0\otimes\RRR}\RRR$ which is a concrete model for the derived smash $\RRR\smsh^L_{\RRR\smsh\RRR}\RRR$ and $dS^1\otimes\RRR=((I\coprod I)\otimes\RRR)\smsh_{(S^0\coprod S^0)\otimes\RRR}S^0\otimes\RRR\cong \RRR\smsh_{\RRR\smsh\RRR}I\otimes(\RRR\smsh\RRR)\smsh_{\RRR\smsh\RRR}\RRR$. 
\vspace{0.1in}

 Let $\RRR\to\AAA$ be a cofibration of cofibrant commutative motivic ring spectra.  Let $X\otimes^\RRR\AAA$ be the tensor in the category of commutative $\RRR$-algebras of the space $X$ (all smashes involved are over $\RRR$). If $\MMM$ and $\NNN$ are $\AAA$-modules, then the derived smash $\MMM\smsh^L_\AAA\NNN$ is conveniently modeled as $\MMM\smsh_\AAA(I\otimes^\RRR\AAA)\smsh_\AAA\NNN$, often referred to as the ``two-sided bar construction over $\RRR$''.  Note that this does not depend on $\RRR$, in the sense that the map $\MMM\smsh_\AAA(I\otimes\AAA)\smsh_\AAA\NNN\to\MMM\smsh_\AAA(I\otimes^\RRR\AAA)\smsh_\AAA\NNN$ is an equivalence.
 In the special case $\AAA=S^0\otimes\RRR=\RRR\smsh\RRR$ we get an identification between the tensor with the subdivided circle and the bar construction
 $S^1\otimes\RRR\cong \RRR\smsh_\AAA (I\otimes^\RRR\AAA)\smsh_\AAA\RRR$
  and 
 $dS^1\otimes\RRR\cong \RRR\smsh_\AAA (I\otimes\AAA)\smsh_\AAA\RRR$.
 If one wishes to write the comultiplication
 $$\psi\colon\xymatrix{S^1\otimes\RRR&\ar[l]_\sim dS^1\otimes\RRR\ar[r]^-{\nabla\otimes 1}& (S^1\vee S^1)\otimes\RRR\ar[r]^-\cong&
 (S^1\otimes\RRR)\smsh_\RRR(S^1\otimes\RRR)}$$
 in terms of the bar construction, a concrete way is to use the equivalence $I\otimes\AAA\to\AAA$ and the augmentation $I\otimes\AAA\to\RRR$ as in the diagram
$$
\xymatrix{
   \RRR\smsh_\AAA (I\otimes\AAA)\smsh_\AAA\RRR&\ar[l]_-{\sim}
 \RRR\smsh_\AAA (I\otimes\AAA)\smsh_\AAA (I\otimes\AAA)\smsh_\AAA (I\otimes\AAA)\smsh_\AAA\RRR\ar[d]
   \\
   (\RRR\smsh_\AAA (I\otimes\AAA)\smsh_\AAA \RRR)\smsh_\RRR(\RRR\smsh_\AAA (I\otimes\AAA)\smsh_\AAA\RRR)&\ar[l]_-\cong\RRR\smsh_\AAA (I\otimes\AAA)\smsh_\AAA \RRR\smsh_\AAA (I\otimes\AAA)\smsh_\AAA\RRR.
   }
$$
This formula only uses the augmentation $\AAA\to\RRR$ and not specifically that $\AAA=\RRR\smsh\RRR$. 
One may replace the $\otimes$ by $\otimes^\RRR$ if convenient.

\subsection{Some classes coming from the dual motivic Steenrod algebra}\label{sec:subsec:dualSteenrod}
Let $\AAA_{\star}=\pi_\star(\M\FF_p\smsh\M\FF_p)$ be the dual motivic Steenrod algebra of our ground field $F$ at $p$,
\begin{equation}
\label{equation:Aingeneral}
\AAA_{\star}
=
\begin{cases}
\MM_{\star}[\xi_{i},\tau_{i}]_{i\geq 0}/(\tau_{i}^{2}-\rho(\tau_{i+1}-\tau_{0}\xi_{i+1})-\tau\xi_{i+1})& p=2 \\
\MM_{\star}[\xi_{i}]_{i\geq 0}\otimes_{\MM_{\star}} \LAMBDA_{\MM_{\star}}(\tau_{i})_{i\geq 0} & p\ne 2
\end{cases}
\end{equation}
(where $\MM_{\star}$ is the mod-$p$ motivic homology ring of $F$; $\tau$ and $\rho$ are discussed below), 
 whose Hopf algebroid structure 
is given in \cite[\S5.1]{MR3730515}, \cite[\S12]{MR2031198}.
Our notation indicates that $\tau_{i}$ is an exterior class when $p\neq 2$. 
By convention we set $\xi_{0}=1$.
The bidegrees of the generators in \eqref{equation:Aingeneral} are given by 
\[
\vert\xi_{i}\vert=(2p^i-2,p^{i}-1),\qquad
\vert\tau_{i}\vert=(2p^i-1,p^{i}-1)
\] 
The coproducts of the generators are defined by 
\begin{equation}
\label{equation:coproducts}
\psi\xi_{i} =\sum_{j=0}^{i}\xi_{i-j}^{2^{j}}\otimes\xi_{j}, \qquad
\psi\tau_{i}=\tau_{i}\otimes 1+\sum_{j=0}^{i}\xi_{i-j}^{2^{j}}\otimes\tau_{j}.
\end{equation}
The left unit is the canonical inclusion.
When $p=2$,  
the right unit is determined by 
\[
\eta_{R}(\rho)=\rho,\qquad \eta_{R}(\tau)=\tau+\rho\tau_{0}
\]
for the canonical classes $\tau\in\MM_{0,-1}\cong\mu_{2}(F)$ and 
$\rho\in\MM_{-1,-1}\cong F^{\times}/(F^{\times})^{2}$.
The mod $2$ Bockstein on $\tau$ equals $\rho$.
While $\tau$ is always nontrivial 
--- being the class of $-1\in\mu_{2}(F)$ --- 
we have $\rho=0$ if $\sqrt{-1}\in F$. 
The graded mod-$2$ Milnor $K$-theory ring $k_{\ast}^{M}\subseteq\MM_{\star}$ of the base 
field $F$ is comprised of primitive elements.
The element $\tau$ is not primitive in general.
If $F$ contains a primitive $p$th root of unity so that $\MM_{0,-1}\cong \ZZ/p\{\tau\}$,
then $\MM_{\star}\cong k^{M}_{\ast}[\tau]$ by the norm residue isomorphism \cite{Voevodskymod2}, \cite{MR2811603}.
We shall also use the antipodal generators 
\begin{equation}
\label{equation:antipodalgenerators}
c(\tau_i)=-\tau_i + \sum_{j=0}^{i-1} \xi_{i-j}^{p^j}c(\tau_j), 
\qquad
c(\xi_i) = -\xi_i + \sum_{j=1}^{i-1} \xi_{i-j}^{p^j}c(\xi_j)
\end{equation}
detailed in \cite[\S5]{MR3730515}. 
{\it For legibility, 
we will abuse notation by implicitly using the antipodal classes \eqref{equation:antipodalgenerators} 
in our computations.}
Voevodsky defines in \cite[\S3.1]{MR1977582} the mod-$p$ rigid motivic Steenrod algebra
\begin{equation}
\label{equation:rigidA}
\AAA_{\star}^{\textrm{rig}}
:=
\bigotimes_{i\geq 0}\SSS_{\FF_{p}}(\xi_{i+1})\otimes\LAMBDA_{\FF_{p}}(\tau_{i}).
\end{equation}

The equation \eqref{equation:coproducts} gives the coproducts of the generators.
For $p\neq 2$ this is the dual topological Steenrod algebra at $p$.

\begin{remark}
\label{remark:rigidity}
Suppose $\overline{F}$ is an algebraically closed field of positive characteristic $\neq p$.
Its ring of Witt vectors $W(\overline{F})$ is a henselian local ring with residue field $\overline{F}$.
Let $\overline{K}$ denote an algebraic closure of the quotient field $K$ of $W(\overline{F})$. 
We note that $\overline{K}$ has characteristic zero.
The natural maps 
$$
\overline{K}
\leftarrow  
W(\overline{F})
\rightarrow
\overline{F}
$$
induce isomorphisms on $\MM_{\star}$ and $\AAA_{\star}$ according in \cite[\S 4,5,6]{zbMATH06698209}.
These algebra isomorphisms preserve the classes $\tau_{i}$ and $\xi_{i}$.
Moreover, 
$\MM_{\star}$ and $\AAA_{\star}$ are invariant under extensions of algebraically closed fields 
of characteristic zero.
\end{remark}

The structure of the dual Steenrod algebra has some direct consequences for motivic Hochschild homology.
Recall the suspension operation
$$
\xymatrix{\sigma\colon S^1_+\smsh(\M\FF_p\smsh\M\FF_p)\ar[r]& 
\M\FF_p\smsh\MHH(\FF_p)\ar[r]^-{\text{mult.}}&\MHH(\FF_p)}.
$$
Here, 
$s_1\in H_1(S^1_+;\FF_p)$ is the standard generator and if $\zeta\in\pi_{s,w}(\M\FF_p\smsh\M\FF_p)$, 
we let ``$\sigma\zeta$'' denote the ``homology suspension'' of \(\zeta\), namely the image of \(s_1\zeta\) in $\MHH_{s+1,w}(\FF_p)$ and also in
$\pi_{s+1,w}(\M\FF_p\smsh\MHH(\FF_p))$.

\begin{lemma}
  \label{lem:extension}
  In the motivic Hochschild homology $\MHH_\star(\FF_p)$ we have the relations
  $$\tau^{p-1}\sigma\tau_{i+1}=(\sigma\tau_{i})^p,\qquad
  \tau^{p-1}\sigma\xi_{i+1}=\rho\sigma\tau_{i+1}$$
  for all $i\geq 0$ (where $\rho:=0$ for odd primes $p$).
\end{lemma}
\begin{proof}
  Both relations are shown already in homology from which the homotopy versions follow by the $\M\FF_p$-algebra structure $\text{mult}\colon\M\FF_p\smsh\MHH(\FF_p)\to\MHH(\FF_p)$ (splitting the inclusion of homotopy in homology).
  
  We first show $$\tau^{p-1}(1\smsh\sigma^+)_*(s_1\otimes\tau_{i+1})=[(1\smsh\sigma^+)_*(s_1\otimes\tau_{i})]^p$$
  in the homology $\MM_\star$-algebra $\pi_{\star}(\M\FF_p\smsh \MHH(\FF_p))$.
  Let $\RRR$ be a commutative $\M\FF_p$-algebra and let $\alpha\colon S^{s,w}\to\RRR$ represent a class in $\pi_{s,w}\RRR$.  
  Commutative motivic ring spectra are $E_\infty$-algebras and the induced composite
  $$\xymatrix{
     \M\FF_p\smsh((E\Sigma_p)_+\smsh_{\Sigma_p} (S^{s,w})^{\smsh p})\ar[r]^-\alpha&
 (E\Sigma_p)_+\smsh_{\Sigma_p} \RRR^{\smsh_{\M\FF_p} p}\ar[r]^-{\text{mult}}&
 \RRR
}$$
(where $E\Sigma_p$ is the nerve of the translation category of the symmetric group $\Sigma_p$ on $p$ letters)
gives us a ``power operation''
$$P(\alpha)\colon\pi_\star(\M\FF_p\smsh((E\Sigma_p)_+\smsh_{\Sigma_p}  (S^{s,w})^{\smsh p}))\to \pi_\star\RRR.$$
The image under $P(\alpha)$ of generators may be called Dyer-Lashof operations on $\alpha$.
Precomposing with 
$$H_*(B\Sigma_p;\FF_p)\otimes H_*(S^{s,w};\M\FF_p)\to \pi_\star(\M\FF_p\smsh((E\Sigma_p)_+\smsh_{\Sigma_p}  (S^{s,w})^{\smsh p}))
$$
defined on the chain level as the diagonal
$$C_*(E\Sigma_p,\FF_p)/\Sigma_p\otimes C_*(S^{s,w};\M\FF_p)\to
C_*(E\Sigma_p,\FF_p)\otimes_{\Sigma_p} C_*(S^{s,w};\M\FF_p)^{\otimes_{\FF[\tau]} p}
$$
and evaluation at the classical choice of generator of $H_i(BC_p;\FF_p)\subseteq H_i(\Sigma_p,\FF_p)$ gives us 
the (topological) Dyer-Lashof operation $Q_i(\alpha)\in\pi_{i+ps,pw}\RRR$ on $\alpha$.  We do the usual shift to upper indexing with $Q^r(\alpha)\overset\cdot=Q_{(2r-s)(p-1)}(\alpha)$  (for $p$ odd; $Q^r(\alpha)=Q_{r-s}(\alpha)$ for $p=2$) so that $Q^0(\alpha)=\alpha$ and $Q^s(\alpha)=\alpha^p$ when $2r=s$ (for $p$ odd; $r=s$ for $p=2$).
\vspace{0.1in}

If in \Cref{sec:compsimp} we set $\RRR=\M\FF_p$ and $\AAA=\RRR\smsh\RRR$ and let $X$ be any space with finite basis $\{x\}$ for the homology $H_*(X;\FF_p)$ and recall that $\alpha^\RRR\colon\AAA\to(X\otimes^\RRR\AAA)^X$ and $\nu\colon \RRR^X\smsh_X(X\otimes^\RRR\AAA)\to X\otimes^\RRR\AAA$ were homomorphisms of commutative $\RRR$-algebras we get that
$$Q^s\nu^{-1}_*\alpha^\RRR_*=\nu^{-1}_*\alpha^\RRR_*Q^s$$
and so
$$\sum_x\sum_{a+b=s}Q^ax^*\otimes Q^b\sigma^\RRR(x\otimes a)=\sum_xx^*\otimes\sigma^\RRR(x\otimes Q^s a).$$
When $X=S^1$ the Dyer-Lashof operations $Q^a(x^*)$ are trivial for $a\neq0$ and so we get
$$
Q^s\sigma^\RRR(x\otimes a)=\sigma^\RRR(x\otimes Q^s a).
$$
Restricting to the generator $x=s_1\in H_1(S^1;\FF_p)$ and multiplying down to homotopy, 
we get the crucial formula
$$
Q^s\sigma=\sigma Q^s\colon\pi_{s,w}(\M\FF_p\smsh\M\FF_p)\to \MHH_{s+1+2s(p-1),pw}(\FF_p).
$$
By construction, 
the power operations are preserved under base change.  
Over any algebraically closed field, we claim there is a relation 
\begin{equation}
\label{equation:relation1}
\tau^{p-1}\tau_{i+1}=Q^{p^i}\tau_i\in\pi_{2p^{i+1}-1,p^{i+1}-p}(\M\FF_p\smsh\M\FF_p).
\end{equation}
By \Cref{remark:rigidity} the relation holds over $\bar F$ if and only if it does so over $\bar K$. 
Consequently, 
by rigidity is suffices to know that the relation holds over the complex numbers, 
which follows by Betti realization to the topological situation
(the motivic correction factor $\tau^{p-1}$ ensures the weights agree).
Thus, 
for the antipodal classes \eqref{equation:antipodalgenerators}, 
we obtain the formula
\begin{equation}
\label{equation:relation2}
\tau^{p-1}\sigma\tau_{i+1}=\sigma Q^{p^i}\tau_i=Q^{p^i}\sigma\tau_i= (\sigma\tau_{i})^p.
\end{equation}

We use this result to prove the vanishing of $(1\smsh\sigma^+)_*(s_1\otimes\tau^{p-1}\xi_{i+1})$.
  Let $\beta\colon\M\FF_p\to S^1\smsh\M\FF_p$ be the $p$-Bockstein, i.e.,~the $\M\FF_p$-linear boundary map in the fiber sequence of Eilenberg-MacLane spectra associated with the short exact sequence $p\ZZ/p^2\subseteq\ZZ/p^2\to\ZZ/p$.
  For any commutative ring spectrum $\RRR$ the map $(\beta\smsh 1)_*\colon\pi_\star(\M\FF_p\smsh\RRR)\to \pi_*(S^1\smsh\M\FF_p\smsh\RRR)$ is a derivation (since $p\ZZ/p^2\subseteq\ZZ/p^2$ is a square zero ideal) and as usual we allow ourselves the shorthand $\beta$ for $(\beta\smsh1)_*$.
By construction of the motivic Steenrod algebra, see \cite[\S 5]{MR3730515}, \cite[\S 9]{MR2031198}, 
the generators in the dual motivic Steenrod algebra $\AAA_\star=\pi_\star(\M\FF_p\smsh\M\FF_p)$ are connected via $$\xi_{i+1}=\beta\tau_{i+1}.$$ 
  Obviously, the diagram
  $$\xymatrix{\M\FF_p\smsh S^1_+\smsh\M\FF_p\ar[r]^{1\smsh\sigma^+}\ar[d]^{\beta\smsh 1}&\M\FF_p\smsh \MHH(\FF_p)\ar[d]^{\beta\smsh 1}\\
    S^1\smsh\M\FF_p\smsh S^1_+\smsh\M\FF_p\ar[r]^{1\smsh\sigma^+}&S^1\smsh\M\FF_p\smsh \MHH(\FF_p)
  }$$
  commutes, and since the power operations commute with $\sigma^{\M\FF_p}=1\smsh\sigma^+$ we get for $p$ odd (where $\beta\tau=0$) 
  that
  \begin{align*}
    0&=\text{mult.}\beta[\sigma^{\M\FF_p}_*(s_1\otimes\tau_i)]^p
     =\text{mult.}\beta\sigma^{\M\FF_p}_*(s_1\otimes\tau^{p-1}\tau_{i+1})\\
     &=\sigma\beta\tau^{p-1}\tau_{i+1}=\sigma\tau^{p-1}\xi_{i+1}
    =\tau^{p-1}\sigma\xi_{i+1}.
  \end{align*}
  For $p=2$ we'll see that we can easily read the last formula directly from the $d^1$-differentials in the Tor-spectral sequence, but we may also use the Bockstein and compute
  \begin{align*}
    0&=\text{mult.}\beta[\sigma^{\M\FF_2}_*(s_1\otimes\tau_i)]^2
       =\text{mult.}\beta\sigma^{\M\FF_2}_*(s_1\otimes\tau\tau_{i+1})\\
       &=\sigma(\beta(\tau\tau_{i+1}))=
    \sigma(\beta(\tau)\tau_{i+1}+\tau\beta(\tau_{i+1}))=\sigma(\rho\tau_{i+1}+\tau\xi_{i+1})=\rho\sigma\tau_{i+1}+\tau\sigma\xi_{i+1}.
  \end{align*}
\end{proof}

\subsection{Tor spectral sequence for motivic Hochschild homology}
\label{subsction:tkssfmhh}

A motivic spectrum is cellular if it belongs to the smallest full subcategory of the 
stable motivic homotopy category, which is closed under homotopy colimits and contains 
the motivic spheres $S^{p,q}$ for all $p, q \in \ZZ$, 
see \cite[\S2.8]{MR2153114}.
The cellularity assumption is central in motivic homotopy theory, 
see, e.g., 
\cite[\S2.3]{zbMATH07003144}.
It is, moreover, needed for running the motivic Tor spectral sequence.
\vspace{0.1in}

We begin by relating the integral Tor spectral sequence to the bar construction.
Our setup is a map of motivic ring spectra $\RRR\rightarrow\QQQQ$ and an $\QQQQ$-$\RRR$ bimodule $\MMM$.
We assume that $\RRR$ is a commutative motivic ring spectrum and $\AAA=\QQQQ\wedge\RRR$ is a cofibrant $\RRR$-algebra.
Then the derived smash product $\MMM\wedge_{\AAA}\QQQQ$ in \eqref{equation:mhh2} is the realization of the 
simplicial object 
$$
\{[s]\mapsto \MMM\smsh_{\RRR} \AAA^{\smsh_{\RRR}^{s}}\smsh_{\RRR} \QQQQ\}. 
$$
The skeletal filtration yields the $E^1$ page of the Tor spectral sequence, 
which --- if $\pi_{\star}\AAA$ is flat over $\pi_{\star}\RRR$ --- 
takes the form
$$
E^1_{s,\star}
=
B_s(\pi_{\star}\MMM,\pi_{\star}\AAA,\pi_{\star}\QQQQ)
=
\pi_{\star}\MMM
\otimes_{\pi_{\star}\RRR}\pi_{\star}\AAA^{\otimes_{\pi_{\star}\RRR}^{s}}
\otimes_{\pi_{\star}\RRR}
\pi_{\star}\QQQQ.
$$
It is conventional to denote the generators of the bar complex by $[m_0|a_1|\dots|a_s|n_{s+1}]$.
When $\MMM=\QQQQ=\RRR$ we abbreviate $[1|a_1|\dots|a_s|1]$ to $[a_1|\dots|a_s]$.
The homology of $E^1_{s,\star}$ computes the $E^{2}$ page of the Tor spectral sequence \eqref{equation:bss}.
We recall the $d^{1}$ differential is given by the alternating sum of the face maps
$$
d_{i}[m_0|a_1|\dots|a_s|n_{s+1}]
=
\begin{cases}
  [m_0\cdot a_1|a_2|\dots|a_s|n_{s+1}]&i=0\\
  [m_0|a_1|\dots|a_i\cdot a_{i+1}|\dots|a_s|n_{s+1}]&0<i<s\\
  [m_0|a_1|\dots|a_{s-1}|a_s\cdot n_{s+1}]&i=s.
\end{cases}
$$

If $\AAA$ is commutative 
and the modules $\MMM$ and $\QQQQ$ are commutative $\AAA$-algebras, 
then the skeletal filtration of the simplicial object 
$\{[s]\mapsto \MMM\smsh_{\RRR} \AAA^{\smsh_{\RRR}^{s}}\smsh_{\RRR} \QQQQ\}$ is isomorphic to $(\MMM\smsh_\RRR\QQQQ)\otimes_\AAA (S^1\otimes \AAA)$ \emph{in the category of commutative $\RRR$-algebras}. This is 
isomorphic to the more Hochschild-homology-looking $\{[s]\mapsto(\MMM\smsh_\RRR\QQQQ)\smsh\QQQQ^{\smsh s}\}$, and the filtration arises from the simplicial circle.
Hence the Tor spectral sequence \eqref{equation:bss} is a $\pi_{\star}\RRR$-algebra spectral sequence with the 
multiplicative structure on the $E^{1}$ page given by the shuffle product introduced by Eilenberg-Mac Lane 
\cite{zbMATH03080175}.

\begin{proposition}
\label{prop:bss}
Suppose $\MMM$, $\RRR$, and $\QQQ$ are cellular motivic spectra.
Then 
the skeletal filtration of the simplicial circle gives rise to a strongly convergent trigraded Tor spectral sequence 
\begin{equation}
\label{equation:bss}
E^{2}_{h,t,w}
=
\TOR^{\pi_{\star}(\QQQ\wedge\RRR^{\op})}_{h,t,w}(\pi_{\star}(\MMM),\pi_{\star}(\QQQ))
\Rightarrow
\MHH_{h+t,w}(\RRR;\MMM).
\end{equation}
Here, 
$h$ is the homological grading on the torsion product and $(t,w)$ is the internal grading 
for the bigraded motivic homotopy groups in topological degree $t$ and weight $w$.
The differentials are of the form 
\[
d^{r}
\colon 
E^{r}_{\ast,\ast,\ast}\rightarrow E^{r}_{\ast-r,\ast+r-1,\ast}.
\]
If $\RRR$ is commutative and $\QQQ$ and $\MMM$ are commutative $\RRR$-algebras, then the Tor spectral sequence is a spectral sequence of $\RRR$-algebras.  The pinch map on the circle induces the Hopf-algebra structure on the torsion groups on the $E^2$ page. If $E^2$,\dots,$E^r$ are all flat over $\RRR_\star$ for $2\leq r$, then the $E^r$ page inherits an $\RRR$-Hopf algebra structure; in particular, the $r$-th differential $d^r\colon E^r\to E^r$ satisfies the ``co-Leibniz'' rule in the sense that it commutes with the coproduct $\psi\colon E^r\to E^r\otimes E^r$.
\end{proposition}
\begin{proof}
This follows from (\ref{equation:mhh2}),
\cite[\S4]{zbMATH02221879}, 
\cite[Proposition 7.7]{MR2153114}, 
\cite[\S2]{zbMATH03730966}.
\end{proof}

The suspension map $\sigma\colon \RRR\smsh S^1\smsh\RRR\to S^1\otimes\RRR$ has a simple interpretation under 
the isomorphism
$$S^1\otimes\RRR\cong\{[s]\mapsto (\RRR\smsh_{\RRR}\RRR)\smsh(\RRR\smsh\RRR)^{\smsh_{\RRR}s}\}
$$
of the Tor-interpretation: it is the map from $S^1\smsh(\RRR\smsh\RRR)=\{[s]\mapsto\bigvee_{\{1,\dots,s\}}(\RRR\smsh\RRR)\}$ sending the $i$th summand to the inclusion on the $i$th factor (and units elsewhere).  In particular, if $x\in\pi_{d,w}(\RRR\smsh\RRR)$, then $\sigma_*x\in \pi_{d+1,w}(S^1\otimes\RRR)$ is the class represented by $[x]\in E^1_{1,d,w}$.
\vspace{0.1in}

The Hopkins-Morel equivalence shown by Hoyois \cite[Proposition 8.1]{MR3341470} implies the cellularity assumption 
in \Cref{prop:bss} holds for $\MMM=\RRR=\QQQ=\M\FF_{p}$ since the base scheme $F$ is a field of exponential 
characteristic $\cc(F)\neq p$.
In this case, 
we have the Tor spectral sequence
\begin{equation}
\label{equation:kssMFp}
E^{2}_{h,t,w}
=
\TOR^{\AAA_{\star}}_{h,t,w}(\MM_{\star},\MM_{\star})
\Rightarrow
\MHH_{h+t,w}(\M\FF_{p}).
\end{equation}

\begin{remark}
\label{remark:rigidity2}
By \Cref{remark:rigidity2} and \eqref{equation:kssMFp} it follows that, 
for algebraically closed fields, 
$\MHH_{\star}(\M\FF_{p})$ is independent of the exponential characteristic $\neq p$.
\end{remark}

\subsection{Torsion products}
For reference we record some basic facts about the structure of $\Tor$ in simple situations.
If $(\x_{i})$ is a basis for an $\FF_{p}$-vector space,
we denote the corresponding divided power, exterior, and symmetric algebras by $\GAMMA_{\FF_{p}}(\x_{i})$, 
$\LAMBDA_{\FF_{p}}(\x_{i})$, and $\SSS_{\FF_{p}}(\x_{i})$, 
respectively.
Let $\gamma_{n}\x$ denote the $n$th divided power of a class $\x$ in degree $d$.
Then the graded divided power algebra $\GAMMA_{\FF_{p}}(\x)$ is generated by elements $\gamma_{p^{j}}\x$ in degree 
$dp^{j}$ subject to the relations $\gamma_{0}\x=1$, 
$\gamma_{1}\x=\x$, and
\[ 
\gamma_{m}\x\cdot\gamma_{n}\x
=\left (\begin{array}{c} m+n \\ [1ex] m \end{array}\right)
\gamma_{m+n}\x.
\]

The symmetric algebra functor $\SSS_{\FF_{p}}(-)$ is left adjoint to the forgetful functor from 
$\FF_{p}$-algebras to $\FF_{p}$-modules, 
and the symmetric algebra $\SSS_{\FF_{p}}(\x_{i})$ is canonically isomorphic to a polynomial ring.
We let $\PPP^{h}_{\FF_{p}}(\x)$ denote the height $h$ truncated polynomial ring $\PPP_{\FF_{p}}(\x)/\x^{h}$.
With these definitions, 
an exercise in binomial coefficients shows there is an isomorphism of algebras
\begin{equation}
\label{equation:gammaisotruncated}
\GAMMA_{\FF_{p}}(\x)
=
\FF_{p}\{\gamma_{p^{j}}\x\vert j\geq 0\}
\cong
\bigotimes_{j\geq 0}\PPP^{p}_{\FF_{p}}(\gamma_{p^{j}}\x).
\end{equation}

We shall make repeatedly use of the following torsion product computations,
see \cite[\S6]{zbMATH03730966}.
\begin{lemma}
\label{lemma:torsionproducts}
\begin{enumerate}
\item[(i)]
For the symmetric algebra $\SSS_{\FF_{p}}(\x)$ on a generator $\x$ in even degree $d$,
there is an $\FF_{p}$-bialgebra isomorphism 
\begin{equation*}
\TOR^{\SSS_{\FF_{p}}(\x)}_{\ast}(\FF_{p},\FF_{p})
\cong
\LAMBDA_{\FF_{p}}(\sigma\x).
\end{equation*}
Here, 
$\sigma\x$ in degree $(1,d)$ is a coalgebra primitive represented in the bar complex by $[\x]$.
\vspace{0.1cm}
\item[(ii)] 
For the exterior algebra $\LAMBDA_{\FF_{p}}(\x)$ on a generator $\x$ in odd degree $d$,
there is an $\FF_{p}$-bialgebra isomorphism 
\begin{equation*}
\TOR^{\LAMBDA_{\FF_{p}}(\x)}_{\ast}(\FF_{p},\FF_{p})
\cong
\GAMMA_{\FF_{p}}(\sigma\x).
\end{equation*}
Here, $\gamma_j\sigma\x$ in degree $(j,dj)$ is represented in the bar complex by $[\x\vert\x\vert\dots\vert\x]$ and has coproduct
\begin{equation*}
\psi\gamma_{k}\sigma\x
=
\underset{i+j=k}{\sum}
\gamma_{i}\sigma\x
\otimes
\gamma_{j}\sigma\x.
\end{equation*}
\end{enumerate}  
\end{lemma}

\begin{remark}
  \label{rem:firstdiff}
  As an example, let us reconstruct a direct proof of the equation 
  $\tau\sigma\xi_{i+1}=\rho\sigma\tau_{i+1}\in\MHH_{\star}(\FF_2)$ of \Cref{lem:extension}.  Consider the $E^1$ page
  $$E^1_{s,\star}=\AAA_\star^{\otimes_{\MM_\star}s}\cong\MM_\star\otimes_{\MM_\star}\AAA_\star^{\otimes_{\MM_\star}s}\otimes_{\MM_\star}\MM_{\star}
  $$
  of the spectral sequence for $\MHH(\FF_2)$, where $\AAA_\star=\pi_\star(\M\FF_2\smsh\M\FF_2)$ is the dual Steenrod algebra.  Then
  $$d^1[\tau_i|\tau_i]=[\tau^2_i]=\tau[\xi_{i+1}]+\rho[\tau_{i+1}]+\rho[\tau_0\xi_{i+1}]$$
  and $d^1[\tau_0|\xi_{i+1}]=[\tau_0\xi_{i+1}]$, so that $\tau[\xi_{i+1}]+\rho[\tau_{i+1}]$ is a boundary.
  Hence $\tau\sigma\xi_{i+1}=\rho\sigma\tau_{i+1}$.

  With the notation $\gamma_j\sigma\tau_i=[\tau_i|\dots|\tau_i]\in E^1_{j,\star}$ and $\sigma x=[x]$ the shuffle product yields an explicit formula for the $d^{1}$ differentials
  \begin{align*}
        d^1\gamma_{j+2}\sigma\tau_i&=\sum_{a=1}^{j-1}[\tau_i|\dots|\tau_i^2|\dots|\tau_i]=
                                     [\tau_i^2]\cdot[\tau_i\dots|\tau_i]=(\tau[\xi_{i+1}]+\rho[\tau_{i+1}]+\rho[\tau_0\xi_{i+1}])\cdot[\tau_i\dots|\tau_i]\\
    &=(\tau\sigma\xi_{i+1}+\rho\sigma\tau_{i+1}+\rho\sigma(\tau_0\xi_{i+1}))\cdot\gamma_j\sigma\tau_i.
  \end{align*}
  When the ground field contains a square root of $-1$, so that $\rho=0$, we get the formula
  $$d^1\gamma_{j+2}\sigma\tau_i=\tau\sigma\xi_{i+1}\cdot\gamma_j\sigma\tau_i.$$

     Conversely, for odd primes $p$ we can use \Cref{lem:extension} to \emph{deduce} differentials by a simple weight count --- simplifying the corresponding topological argument.
     \Cref{lemma:torsionproducts} tells us that
     $$E^2=\MM_\star\otimes_{\FF_p}
     \bigotimes_{i\geq}\Lambda_{\FF_{p}}(\sigma\xi_{i+1})\otimes\Gamma_{\FF_{p}}(\sigma\tau_i).$$
     We know that $\tau^{p-1}\sigma\xi_{i+1}$ has to be hit by a differential.  When the ground field is algebraically closed, $\MM_\star=\FF_p[\tau]$ with $\tau\in\MM_{0,-1}$.  In this case the source of the differential hitting $\tau^{p-1}\sigma\xi_{i+1}$ must come from linear combinations of monomials in $\sigma\xi_{i}$s and $\gamma_j\sigma\tau_{i}$s of total degree  $2p^{i+1}$ and weight at least $p^{i+1}-p$.  A quick count shows that the only monomial with sufficient weight is $\gamma_p\sigma\tau_i$ and so  we have the equation 
     (described up to a unit in $\FF_p$)
     $$d^{p-1}\gamma_p\sigma\tau_i\overset{\cdot}=\tau^{p-1}\sigma\xi_{i+1}.$$
   \end{remark}

\subsection{A Bockstein type complex}
\label{sec:Dcomplex}
We end the section by doing an entirely algebraic exercise which will be needed later on.  
Let $p$ be any prime and consider the commutative differential graded $\FF_p$-algebra $(C,D)$, where
$$C=\bigotimes_{i\geq 0}\GAMMA_{\FF_{p}}(\bmu_{i})
\otimes 
\LAMBDA_{\FF_{p}}(\blambda_{i+1})$$
and $D\colon C\to C$ is the derivation generated by $D(\gamma_{j+p}\bmu_i)=\blambda_{i+1}\gamma_j\bmu_i$.
As before, 
$B^D=\im D$, $Z^D=\ker D$ and $H^D=Z^D/B^D$ --- the aim of this subsection is to calculate these.
In the application $C$ will be the mod-$\tau^{p-1}$ motivic Hochschild homology of $\FF_p$ 
(the reader may recognize it as $\Tor^{\AAA^{\mathrm{rig}}_\star}_*(\FF_p,\FF_p)$) and $D$ will be derived from a Bockstein.
\vspace{0.1in}

We first fix some notation.
For each non-empty finite set of natural numbers $S\subseteq\NN$,
we choose an element $t_S\in S$ with the property that $t_{S\cup T}\in\{t_S,t_T\}$.  
The minimum, $t_S=\min S$, is a good choice, but there are many others.  
Down the road, such a choice amounts to a particular choice of basis, 
and there is no reason to prefer one over the other, 
except that in concrete examples, 
some can be more convenient.  
If the function $f\colon\NN\circlearrowleft$ has finite non-empty support, 
we write $t_f=t_{\supp f}$.
For every $j\in\NN$, 
let $\delta_j\colon\NN\circlearrowleft$ be the function with $\supp\,\delta_j=\{j\}$ and $\delta_j(j)=1$.

\begin{defn}
\label{def:JandK}
Let $J$ denote the set of pairs $(S,f)$, 
where the function $f\colon \NN\circlearrowleft$ has finite support and $S\subseteq\supp f$.  
The subset $K\subseteq J$ consists of the pairs $(S,f)$, 
where the support of $f$ is non-empty and $S$ does not contain $t_f$.
\end{defn}

\begin{defn}
For $(S,f)\in J$, 
we set
\begin{equation}
\label{equation:modtausubcomplex}
\chi_{S,f}
:=
\left(\prod_{m\in S}\blambda_{m+1}\gamma_{pf(m)-p}\bmu_m\right)
\left(\prod_{n\not\in S}\gamma_{pf(n)}\bmu_n\right)
\in C.
\end{equation}
\end{defn}
In particular, 
$\chi_{\emptyset,0}=1$, 
$\chi_{\emptyset,p^j\delta_n}=\gamma_{p^{j+1}}\bmu_n$ and $\chi_{\{m\},\delta_m}=\lambda_{m+1}$.

We note that
$$D\chi_{S,f}=\sum_{n\in\supp(f)-S}\chi_{S\cup\{n\},f}$$ 
since $D\gamma_n\bmu_{i}=\blambda_{i+1}\gamma_{n-p}\bmu_{i}$.
Next we construct sub-complexes of $(C,D)$.

\begin{defn}
  If $f\colon \NN\circlearrowleft$ has finite support, 
the \emph{associated $f$-cube} is the sub-complex
$$
(C^f,D)
:=
(\bigoplus_{S\subseteq\supp f}
\FF_p\{\chi_{S,f}\},D)
\subseteq (C,D).
$$
If $f=0$, then $C^f=\FF_{p}$ . 
Furthermore, let $Z^f=\ker D\cap C^f$, $B^D=\im D\cap C^f$ and $H^f=Z^f/B^f$.
\end{defn}

Note that if $f=0$, then $Z^f=H^f=\FF_p$.  
Recall the number $t_f\in\supp f$ chosen once and for all (whenever $\supp f$ is non-empty) just before 
Definition~\ref{def:JandK}.
\begin{lemma}
\label{lem:gamma2lambdaacyclic}
  If $f\colon \NN\circlearrowleft$ has finite nonempty support,  then $(C^f,D)$ is contractible so that $H^f=0$.  
  Furthermore, $B^f=Z^f$ is generated by the $D\chi_{S,f}$ with $t_f\notin S\subseteq \supp f$.
\end{lemma}
\begin{proof}
For $N=|\supp f|$ and $k=0,\dots,N$, 
let $C^f_{k}\subseteq C^f$ be the span of the $\chi_{S,f}$ with $k=|S|$.  
From the formula $D\chi_{S,f}=\sum_{n\in\supp(f)-S}\chi_{S\cup\{n\},f}$ we see that the differential restricts 
to a chain complex
$$
(C^f,{D})=\left\{
\xymatrix{C^f_{0}\ar[r]^{D} & 
C^f_{1}\ar[r]^{D} & 
\dots\ar[r]^-{D}&C^f_{N-1}\ar[r]^-{D} &
C^f_{N}}\right\}.
$$
Here the $\FF_{p}$-vector space $C^f_{k}$ is of dimension $\binom Nk$ with basis elements $\chi_{S,f}$, 
where $|S|=k$.
Note that the set 
$$\{{D}\chi_{S,f}\mid k=|S|, t_f\notin S\}$$ is linearly independent because only ${D}\chi_{S,f}$ 
has a nontrivial $\chi_{\{t_f\}\cup S,f}$-coefficient.  
Hence the rank of ${D}\colon C^f_{k}\to C^f_{k+1}$ is at least $\binom{N-1}k$, 
and we deduce that
\begin{align*}
\dim H_k(C^f,{D}) & =
\dim\ker\{{D}\colon C^f_{k}\to C^f_{k+1}\}-\dim\im\{{D}\colon C^f_{k-1}\to C^f_{k}\}\\
& \leq \binom Nk-\binom{N-1}k-\binom{N-1}{k-1}=0
\end{align*}
and so $B^f=Z^f$ is generated by the $D\chi_{S,f}$ with $t_f\notin S\subseteq \supp f$, as claimed.
The calculation works when $k=0$ or $k=N$ 
(but not for $N=0$ since then we cannot choose $t_f$).
\end{proof}

We analyze the multiplicative structure.

\begin{defn}\label{def:somenumbers}
For functions $f,g\colon \NN\circlearrowleft$ with finite support and non-empty finite sets $S,T\subseteq\NN$ 
define $K_{S,T,f,g}\in\FF_{p}$ by
$$
K_{S,T,f,g}
=
\left(\prod_{s\in S}\binom{fs-1+gs}{fs-1}\right)
\left(\prod_{t\in T}\binom{ft+gt-1}{ft}\right)
\left(\prod_{c\notin S\cup T}\binom{fc+gc}{fc}\right)
$$
if $(S,f),(T,g)\in J$ and $S\cap T=\emptyset$, and set $K_{S,T,f,g}=0$ otherwise.
Moreover, 
we define 
$$
\epsilon_{u,S,T,f,g}
=
K_{S\cup\{u\},T\cup\{t_{f+g}\},f+g}
+
K_{S\cup\{t_{f+g}\},T\cup\{u\},f+g}.
$$
\end{defn}
Note that when $(S,f),(T,g)\in J$ and $S\cap T=\emptyset$, 
each factor in the formula $K_{S,T,f,g}$ is $1$ unless the index is in 
$\supp \, f\cap\supp \, g=\supp (f\cdot g)$, 
and so we can restrict to these factors to simplify the calculation.  
We will need $\epsilon_{u,S,T,f,g}$ only in the case when $u\in\supp(f+g)$, 
$u\not=t_{f+g}$ and $u\notin S\cup T$.
\vspace{0.1in}

The following lemma, 
a consequence of the defining relations among divided power generators of $C$, 
explains the relevance of these numbers.
\begin{lemma}
  \label{lem:chismultiply}
  For $(S,f),(T,g)\in J$ we have
$$
\chi_{S,f}\chi_{T,g}=K_{S,T,f,g}\cdot\chi_{S\cup T,f+g}
$$
and if $(S,f),(T,g)\in K$, then 
$$
D\chi_{S,f}\cdot D\chi_{T,g}
=
\sum_{t_{f+g}\not=u\in\supp(f+g)-S\cup T}\epsilon_u\cdot D\chi_{S\cup T\cup\{u\},f+g}.
$$\hfill\qedsymbol
\end{lemma}

\begin{lemma}
\label{lem:fracturebyfs}
The multiplication gives an extra grading indexed by the generators of the commutative differential 
graded sub-algebra $\bigoplus_fC^f\subseteq C$.
In particular, if $f,g\colon \NN\circlearrowleft$ there is a commutative diagram 
$$
\xymatrix{ 
C^f\otimes C^g \ar[r] \ar[d] & 
C\otimes C \ar[d] \\
C^{f+g}\ar[r] & \,C. }
$$
Here the rows are given by the evident inclusion and the columns by multiplication.
The resulting algebra inclusions
$$
\bigoplus_{f\colon \NN\circlearrowleft}C^f\subseteq C
$$ 
and 
$$
{\FF_{p}}[\bmu_{i}]/\bmu_i^p\subseteq C
$$ 
induce isomorphisms of graded commutative $\FF_{p}$-algebras
$$
\bigoplus_{f\colon \NN\circlearrowleft} C^f
\cong 
\bigotimes_{i\geq 0} 
\left(\LAMBDA_{\FF_{p}}(\blambda_{i+1})
\otimes 
\bigotimes_{j> 0}{\FF_{p}}[\gamma_{p^{j}}\bmu_{i}]/(\gamma_{p^{j}}\bmu_{i})^p\right)
$$
and
$$
C
\cong
\left(\bigotimes_{i\geq 0}{\FF_{p}}[\bmu_{i}]/\bmu_{i}^p\right)
\otimes
\left(\bigoplus_{f\colon \NN\circlearrowleft}C^f\right).
$$
\end{lemma}
\begin{proof}
The multiplicative structure follows from Lemma~\ref{lem:chismultiply}, 
and the last two isomorphisms follow from the fact that a monomial in $C$ does not have any 
$\bmu_i$-factors of the form $\chi_{S,f}$.
 \end{proof}

 \begin{corollary}
   \label{cor:HDandZD}
   As an $\FF_p$-algebra,
   $$H^D\cong\bigotimes_{i\geq0}\FF_{p}[\bmu_{i}]/\bmu_i^p
$$ 
and $Z^D$ is the subalgebra of $C$ generated by the $\bmu_i$ with $i\geq 0$ and the $D\chi_{(S,f)}$ with $t_f\notin S\subseteq \supp f$.  More explicitly, and writing $\x_{S,f}=D\chi_{(S,f)}$,  the relation expressed in Lemma~\ref{lem:chismultiply} gives an isomorphism
$$
Z^D \cong 
\FF_p[\bmu_{i},\x_{S,f}]_{i\in\NN, (S,f)\in K}/(\bmu_{i}^p,\x_{S,f}
\cdot
\x_{T,g}-\sum_{u}\epsilon_u\cdot\x_{S\cup T\cup\{u\},f+g}).
$$
Here, 
$u\in\supp(f+g)-S\cup T$ and $u\neq t_{f+g}$.
\end{corollary}

\section{Motivic Hochschild homology over algebraically closed fields}
\label{section:acf}

In this section we work over an algebraically closed field $F$ of  exponential characteristic $\cc(F)\neq p$. 
Then $\rho=0$ since every unit is a square, 
and 
\begin{equation}
\label{equation:algclosedmodp}
\MM_{\star}
\cong
k^{M}_{\ast}[\tau]
\cong
\FF_{p}[\tau]
\end{equation}
by \cite[Corollary 4.3, p.254]{MR1764203}, 
where $\vert \tau\vert=(0,-1)$.
From \eqref{equation:Aingeneral} and \eqref{equation:algclosedmodp} we deduce
\begin{equation}
\label{equation:Aalgclosed}
\AAA_{\star}
=
\begin{cases}
\FF_{p}[\tau,\xi_{i+1},\tau_{i}]_{i\geq 0}/(\tau_{i}^{2}-\tau\xi_{i+1}) & p=2 \\
\FF_{p}[\tau,\xi_{i+1}]_{i\geq 0}\otimes\LAMBDA_{\FF_{p}}(\tau_{i})_{i\geq 0} & p\neq 2.
\end{cases}
\end{equation}
If $p=2$ and we invert $\tau$ in $\AAA_{\star}$, 
then $\xi_{i}$ is no longer needed as a generator because $\xi_{i+1}=\tau^{-1}\tau_{i}^{2}$:
\begin{equation}
\label{equation:Aalgclosedtauinv}
\pi_\star((\M\FF_p\smsh\M\FF_p)[\tau^{-1}])\cong\AAA_{\star}[\tau^{-1}]
=
\begin{cases}
\FF_{p}[\tau^{\pm1},\tau_{i}]_{i\geq 0} & p=2 \\
\FF_{p}[\tau^{\pm1},\xi_{i+1}]_{i\geq 0}\otimes\LAMBDA_{\FF_{p}}(\tau_{i})_{i\geq 0} & p\neq 2.
\end{cases}
\end{equation}
Likewise, since $\AAA_\star$ is free as an $\FF_p[\tau]$-module, taking the quotient by $\tau^{p-1}$ (for any prime $p$) gives an isomorphism of Hopf algebras
\begin{equation}
\label{equation:Aalgclosedmodtaup-1}
\pi_\star((\M\FF_p\smsh\M\FF_p)/\tau^{p-1})\cong\AAA_{\star}/\tau^{p-1}
=
(\bigotimes_{i\geq 0}\FF_{p}[\xi_{i+1}]\otimes\LAMBDA_{\FF_{p}}(\tau_{i}))
\otimes
\FF_{p,\tau}.
\end{equation}
Here $\FF_{p,\tau}$ is shorthand for $\FF_{p}[\tau]/\tau^{p-1}$.
In \Cref{sec:modtau}, 
we use \eqref{equation:Aalgclosedmodtaup-1} to compute the coefficients of the mod 
$\tau^{p-1}$ reduction of $\MHH(\FF_{p})$.

\subsection{{\'E}tale motivic Hochschild homology}
\label{sec:tauinverted}

We refer to \cite{2020arXiv200304006B}, \cite{2017arXiv171106258E} for $\tau$-self maps and applications 
towards {\'e}tale hyperdescent for motivic spectra.
Suppose $\RRR/p$ is a motivic $E_{\infty}$ ring spectrum defined over an algebraically closed field.
Then the canonical map 
\begin{equation}
\label{equation:periodization}
\RRR/p\to\RRR/p[\tau^{-1}]
\end{equation}
exhibits the $\tau$-periodization as a motivic $E_{\infty}$ ring spectrum under $\RRR/p$;
see \cite[\S12]{2017arXiv171103061B}, \cite[\S8]{2017arXiv171106258E} for recent expositions.
If $\RRR$ happens to be cellular, 
then so is $\RRR/p[\tau^{-1}]$.
Owing to \cite[Theorem 1.2]{2020arXiv200304006B}, 
\eqref{equation:periodization} is an {\'e}tale localization
(the $\rho$-completion in \cite{2020arXiv200304006B} is obsolete over algebraically closed fields, 
and for $p\neq 2$ the {\'e}tale localization involves only the ``+''-part of $\RRR/p$).
We note that \eqref{equation:periodization} induces an isomorphism on $\tau$-inverted homotopy groups.
\vspace{0.1in}

At all primes, 
the $\tau$-periodic mod-$p$ motivic Steenrod algebra agrees with the tensor product of the topological 
mod-$p$ Steenrod algebra with the Laurent polynomial ring $\FF_{p}[\tau^{\pm 1}]$.
This observation implies that after $p$-completion the $\tau$-periodic motivic stable homotopy groups are 
isomorphic to the classical stable homotopy groups with $\tau^{\pm 1}$ adjoined 
\cite{zbMATH05898278}, \cite[\S4]{zbMATH07303324}.
In \Cref{sec:tauinverted}, 
we prove a similar statement for motivic and topological Hochschild homology.
\vspace{0.1in}

We calculate  $\MHH_\star(\FF_p)[\tau^{-1}]\cong\pi_\star(\MHH(\FF_p)[\tau^{-1}])$ directly by the Tor spectral sequence, using the relations and differentials from \Cref{lem:extension} and \Cref{rem:firstdiff} and by appealing to \Cref{equation:Aalgclosedtauinv} and 
the naturally induced equivalence of motivic spectra
\begin{equation}
\label{cor:tauinvertedMHH}
\MHH(\FF_{p})[\tau^{-1}]
=
(\M\FF_{p}\smsh_{\M\FF_{p}\smsh\M\FF_{p}}\M\FF_{p})[\tau^{-1}]
\xrightarrow{\simeq} 
\M\FF_{p}[\tau^{-1}]\smsh_{(\M\FF_{p}\smsh\M\FF_{p})[\tau^{-1}]}\M\FF_{p}[\tau^{-1}].
\end{equation}
As before we set $\mu_{i}:=\sigma\tau_i$ and $\lambda_{i}:=\sigma\xi_{i}$.

\begin{lemma}
\label{lemma:Einfinitytauinverted}
The Tor spectral sequence of $\MHH(\FF_{p})[\tau^{-1}]$ collapses at the $E^{p}$ page and 
$$
E^{\infty}
=(\bigotimes_{i\geq 0}\LAMBDA_{\FF_{p}}(\mu_{i}))[\tau^{\pm 1}].
$$
For $p$ odd the only nonzero differentials $d^r$ for $r>1$ are generated by
\begin{equation}
\label{equation:bdifferentialoddprimes}
d^{p-1}(\gamma_{j}\mu_{i})
\overset{\cdot}=
\tau^{p-1}\lambda_{i+1}\gamma_{j-p}\mu_{i}
\end{equation}
for all $i\geq 0$, $j\geq p$. 
\end{lemma}

\begin{proof}
  \Cref{lemma:torsionproducts}, \Cref{equation:Aalgclosedtauinv}, and \Cref{cor:tauinvertedMHH} yield the $E^{2}$ page.
When $p=2$, 
we have 
\begin{equation}
\label{equation:E2peventauinverted}
E^{2}
=
(\bigotimes_{i\geq 0}\LAMBDA_{\FF_{2}}(\mu_{i}))[\tau^{\pm 1}].
\end{equation}
Since all the $\mu_{i}$s have filtration degree $1$, 
there are no non-trivial differentials and we conclude that $E^{\infty}=E^2$.
When $p$ is odd, 
the $E^{2}$ page takes the form
\begin{equation}
\label{equation:E2poddtauinverted}
E^{2}
=
(\bigotimes_{i\geq 0}\GAMMA_{\FF_{p}}(\mu_{i})
\otimes 
\LAMBDA_{\FF_{p}}(\lambda_{i+1}))
[\tau^{\pm 1}].
\end{equation}
The Tor spectral sequence starts out as an augmented unital 
$\MM_\star[\tau^{-1}]$-Hopf algebra 
since \eqref{equation:E2poddtauinverted} is flat over $\MM_\star[\tau^{-1}]$. 
Arguing as in \cite[\S4]{zbMATH02221879}, 
\cite[\S5]{zbMATH05010546}, \cite[\S1.2]{zbMATH07303328}, 
\cite{zbMATH00224014}, \cite{zbMATH03730966},
we'll see that the non-trivial differentials are as claimed.
More precisely, since the shortest differential in the lowest total degree must go from an algebra generator 
(these lie in filtration powers of $p$) to a coalgebra primitive (these lie in filtration $1$),  
the differentials $d^r$ for $1<r<p-1$ are all zero.
Recall from \Cref{rem:firstdiff} that we established the said differential for $j=p$ integrally:
$d^{p-1}(\gamma_{j}\mu_{i})
\overset{\cdot}=
\tau^{p-1}\lambda_{i+1}$
and we move from there by induction on $j\geq p$ and the coalgebra structure in \Cref{lemma:torsionproducts}; 
this is, for $k\geq 0$, the calculation
\begin{align}
\psi d^{p-1}(\gamma_{p+k}\mu_{i})
-
\psi(\tau^{p-1}\lambda_{i+1}\gamma_{k}\mu_{i})
& =  
(d^{p-1}\otimes 1+1\otimes d^{p-1})\psi(\gamma_{p+k}\mu_{i}) 
\nonumber \\
&\phantom{=}~ -
\tau^{p-1}
(\lambda_{i+1}\otimes 1+1\otimes\lambda_{i+1})
(\Sigma_{a+b=k}\gamma_{a}\mu_{i}\otimes\gamma_{b}\mu_{i}) 
\nonumber \\
& =
(d^{p-1}\otimes 1)(\gamma_{p+k}\mu_{i}\otimes 1)
+
(1\otimes d^{p-1})(1\otimes\gamma_{p+k}\mu_{i}) 
\nonumber \\
&\phantom{=}~ +
\tau^{p-1}
\underset{a+b=p+k; a,b>0}{\sum}
(
\lambda_{i+1}\gamma_{a-p}\mu_{i}\otimes\gamma_{b}\mu_{i}
+
\gamma_{a}\mu_{i}\otimes\lambda_{i+1}\gamma_{b-p}\mu_{i}
)
\nonumber \\
&\phantom{=}~ -
\tau^{p-1}
\underset{a+b=k}{\sum}(\lambda_{i+1}\gamma_{a}\mu_{i}\otimes\gamma_{b}\mu_{i}
+\gamma_{a}\mu_{i}\otimes\lambda_{i+1}\gamma_{b}\mu_{i})
\nonumber \\
& =
(d^{p-1}\otimes 1)(\gamma_{p+k}(\mu_{i})\otimes 1)
+
(1\otimes d^{p-1})(1\otimes\gamma_{p+k}(\mu_{i}))
\nonumber
\end{align}
shows the difference $d^{p-1}(\gamma_{p+k}\mu_{i})-\tau^{p-1}\lambda_{i+1}\gamma_{k}\mu_{i}$ is a coalgebra primitive;
however, $0$ is the only such element in the given degree.
The remaining algebra generators on the $E^{p}$ page are in filtration degree $\leq 1$, 
and hence $E^{\infty}=E^{p}$.
\end{proof}

\begin{remark}
Alternatively, 
an appeal to rigidity for extensions of algebraically closed fields as in \Cref{remark:rigidity} or \cite{zbMATH05284762} 
(in characteristic zero) reduces to considering the complex numbers.
Over $\CC$, 
the differential \eqref{equation:bdifferentialoddprimes} is forced by B{\"o}kstedt's differential 
$d^{p-1}(\gamma_{j}\mu_{i})=\lambda_{i+1}\gamma_{j-p}\mu_{i}$ in the Tor spectral sequence for $\THH_{\ast}(\FF_{p})$.
In the motivic case, 
the correction term $\tau^{p-1}$ ensures agreement of the weights.
\end{remark}

\begin{theorem}
\label{theorem:tauinvertedMHH}
There are isomorphisms
$$
\MHH_{\star}(\FF_{p})[\tau^{-1}]
\cong
\FF_{p}[\tau^{\pm 1},\mu_{i}]_{i\geq 0}/(\mu_{i}^{p}-\tau^{p-1}\mu_{i+1})
\cong
\FF_{p}[\mu,\tau^{\pm 1}]
\cong
\THH_{\ast}(\FF_{p})[\tau^{\pm 1}].
$$
The generator $\mu$ has bidegree $(2,0)$.
\end{theorem}
\begin{proof}
  \Cref{cor:tauinvertedMHH} 
  shows the $E^{\infty}$ page for $\MHH(\FF_{p})[\tau^{-1}]$ 
is the Laurent polynomials in $\tau$ of the $E^{\infty}$ page for $\THH(\FF_{p})$.
The result now follow from \Cref{lemma:Einfinitytauinverted} and the multiplicative extension 
\begin{equation}
\label{equation:periodizedextension}
\mu_{i}^{p}=
\tau^{p-1}\mu_{i+1}.
\end{equation}
of \Cref{lem:extension}.
\end{proof}

Hence 
all the classes $\mu_{i}\in\MHH(\FF_p)$ are nontrivial and
we may identify the $\tau$-free part in $\MHH_{\star}(\FF_{p})$ with 
\begin{equation}
\label{equation:nontautorsion}
\FF_{p}[\tau,\mu_{i}]_{i\geq 0}/(\mu_{i}^{p}-\tau^{p-1}\mu_{i+1}).
\end{equation}
This is depicted graphically for $p=2$ and $p=3$ in \Cref{fig:MHHF2et} and \Cref{fig:MHHF3et}, respectively.

\begin{figure}
    \centering
    \includegraphics[width=5in]{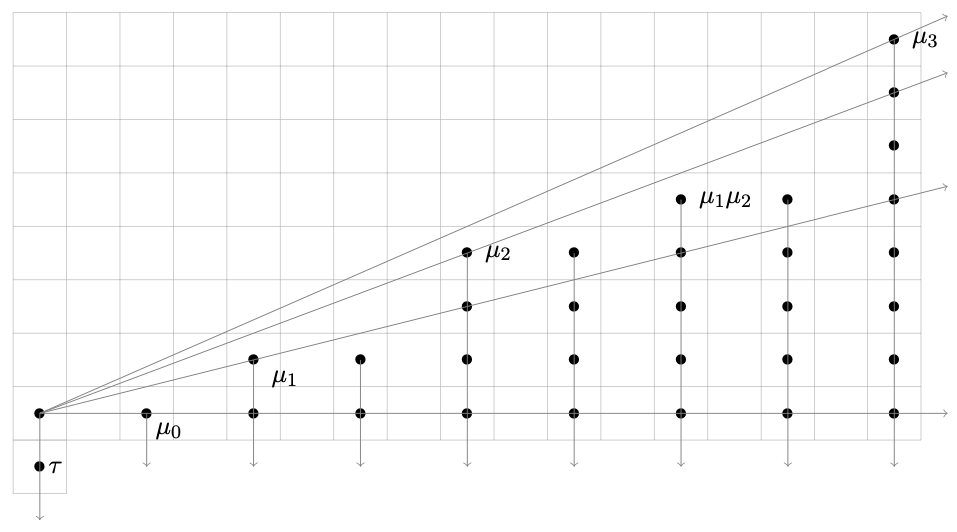}
    \caption{The \'etale motivic Hochschild homology of $\FF_2$. The vertical lines indicate $\tau$-multiplication, while the horizontal and diagonal lines depict powers of $\mu_i$, $i=0,1,2,3$.}
    \label{fig:MHHF2et}
\end{figure}

\begin{figure}
    \centering
    \includegraphics[width=\textwidth]{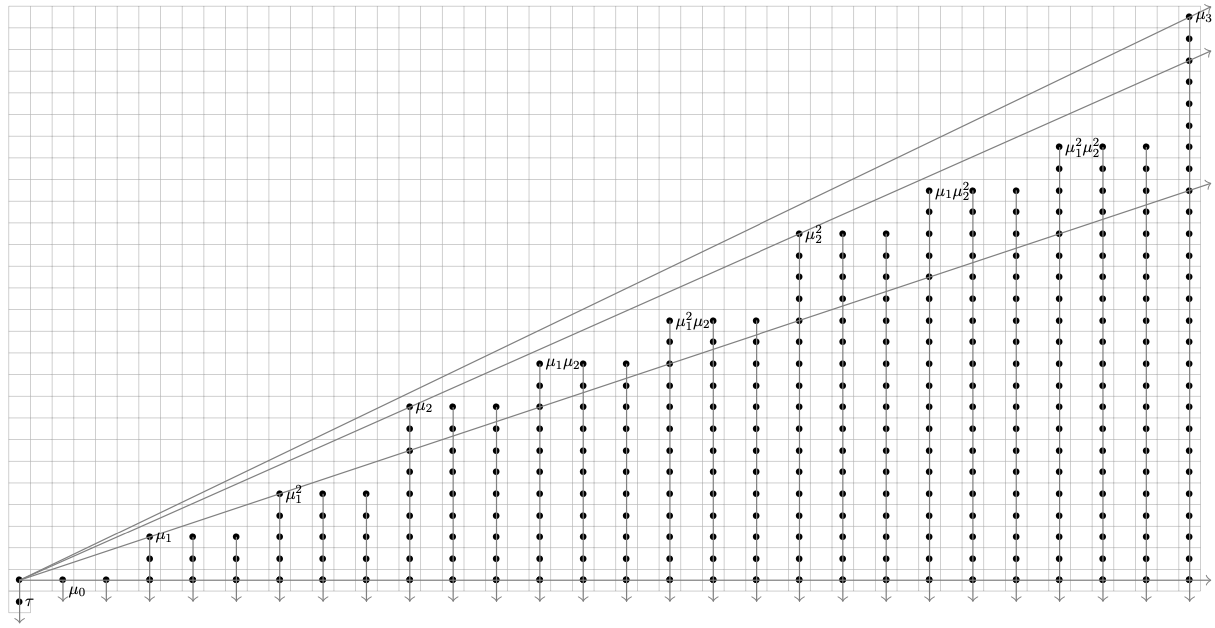}
    \caption{The \'etale motivic Hochschild homology of $\FF_3$ depicted in the same graphical style as \Cref{fig:MHHF2et}.}
    \label{fig:MHHF3et}
\end{figure}

\subsection{Reduced motivic Hochschild homology}\label{sec:modtau}
To proceed to the next step in our strategy for calculating $\MHH(\FF_p)$ over an algebraically closed field $F$ with $\cc(F)\neq p$, 
we form the cofiber of $\tau^{n}$ 
(for our calculations, it suffices to consider $n=p-1$) 
\begin{equation}
\label{equation:cofibertau}
\Sigma^{0,n}\M\FF_{p}
\overset{\tau^{n}}{\rightarrow}
\M\FF_{p}
\rightarrow
\M\FF_{p}/\tau^{n}.
\end{equation}

We thank Markus Spitzweck for informing us that $\M\FF_{p}/\tau^{n}$ is a motivic $E_{\infty}$ ring spectrum 
for all $n\geq 1$.
His argument goes as follows:
$\M\mathbb{Z}$ is strongly periodizable by \cite[Corollary C.3]{zbMATH07015021}.
Thus the mod-$p$ coefficient ring $\MM_{\star}$ is a motivic differential graded algebra, 
i.e., 
a graded $E_{\infty}$ ring spectrum in complexes of $\FF_{p}$-vector spaces.
In fact, 
$\MM_{\star}$ is formal so that $\MM_{\star}/\tau^{n}$ is $E_{\infty}$ over $\MM_{\star}$ for all $n\geq 1$.
This implies the corresponding claim for $\M\FF_{p}/\tau^{n}$. 
In effect, 
let $\mathcal{E}$ be the free $E_{\infty}$ algebra in graded complexes on a generator $\tau$ in degree $(0,1)$. 
Its $0$-truncation, 
with respect to the natural $t$-structure on the derived category of graded abelian groups, 
is the formal model $\FF_{p}[\tau]$.
Thus $\MM_{\star}$ and $\FF_{p}[\tau]$ are equivalent since the natural map $\mathcal{E}\rightarrow \MM_{\star}$ 
is the $0$-truncation.
When $n=1$, 
we also refer to 
Gheorghe \cite{arXiv:1701.04877} for the fact that $\M\FF_{2}\rightarrow\M\FF_{2}/\tau$ 
is a map of motivic $E_{\infty}$ ring spectra.
\vspace{0.1in}

Inserting $\M\FF_{p}/\tau^{n}$ into \eqref{equation:mhh0} yields the derived smash product
\begin{equation}
\label{equation:mhhMFCtau}
\MHH(\FF_{p})/\tau^{n}
\simeq
\M\FF_{p}/\tau^{n}
\wedge_{(\M\FF_{p}\wedge\M\FF_{p})/\tau^{n}}
\M\FF_{p}/\tau^{n}.
\end{equation}
Owing to \eqref{equation:cofibertau} and cellularity of $\M\FF_{p}$, 
see \S\ref{subsction:tkssfmhh}, 
it follows that $\M\FF_{p}/\tau^{n}$ is cellular.
Thus \eqref{equation:mhhMFCtau} gives rise to the Tor spectral sequence
\begin{equation}
\label{equation:kssCtau}
E^{2}_{h,t,w}
=
\TOR^{\AAA_{\star}/\tau^{n}}_{h,t,w}(\MM_{\star}/\tau^{n},\MM_{\star}/\tau^{n})
\Rightarrow
\MHH_{h+t,w}(\FF_{p})/\tau^{n}.
\end{equation}
Recall that $\FF_{p,\tau}$ is shorthand for $\FF_{p}[\tau]/\tau^{p-1}$.
Lemma \ref{lemma:torsionproducts} and \eqref{equation:Aalgclosedmodtaup-1} implies the Tor spectral sequence 
\eqref{equation:bss} for $\MHH_{\star}(\FF_{p})/\tau^{p-1}$ takes the form
\begin{equation}
\label{equation:kssCtaualgebraicallyclosed}
\bar E^{2}_{\ast,\star}
\cong
\left(\bigotimes_{i\geq 0}
\GAMMA_{\FF_{p}}(\sigma\tau_{i})
\otimes 
\LAMBDA_{\FF_{p}}(\sigma\xi_{i+1})\right)
\otimes
\FF_{p,\tau}
\Rightarrow
\MHH_{\star}(\FF_{p})/\tau^{p-1}.
\end{equation}
This is a first quadrant spectral sequence; 
the horizontal direction is the ``filtration'', 
the vertical direction is the ``degree'', 
and every term is graded by ``weight.'' 
Recall that if $x$ has filtration $f_{x}$, degree $d_{x}$ and weight $w_{x}$, 
we write $|x|=(f_{x},d_{x};w_{x})$ so that the differentials take the form
$$
d^r\colon \bar E^r_{f,d;w}\to \bar E^r_{f-r,d+r-1;w}.
$$
In \eqref{equation:kssCtaualgebraicallyclosed}, 
we set $\bmu_{i}:=\sigma\tau_i$ and $\blambda_{i+1}:=\sigma\xi_{i+1}$.  The bar is meant to signify that the generators are 
mod-$\tau^{p-1}$ classes and should not be confused with the conjugate classes.
For these classes, 
we note the degrees 
\begin{enumerate}
\item $|\blambda_{i+1}|=(1,2p^{i+1}-2;p^{i+1}-1)$,
\item $|\gamma_{p^{j}}\bmu_{i}|=(p^{j},2p^{i+j}-p^{j};p^{i+j}-p^{j})$.
\end{enumerate}
Thus for $x=\blambda_{i+1}$ and $x=\gamma_{p^{j}}\bmu_{i}$ we have the congruence $w_{x}\equiv 0\bmod p-1$.
Hence if $x\in \bar E^{2}_{\ast,\star}$ in \eqref{equation:kssCtaualgebraicallyclosed} has weight 
$w_{x}=-n+(p-1)m$, $0\leq n\leq p-1$, 
then $n$ equals $x$'s $\tau$-multiplicity.
Another helpful bookkeeping device for our calculation is the Chow degree of $x$, 
see \cite[Definition 3.1]{2020arXiv201202687B} and \cite[Definition 2.1.10]{2014arXiv1407.8418I} 
for related terminology, 
defined by 
$$
c(x)=f_{x}+2w_{x}-d_{x}
$$
In particular, we have 
\begin{enumerate}
\item $c(\blambda_{i+1})=1+2(p^{i+1}-1)-(2p^{i+1}-2)=1$
\item $c(\gamma_{p^{j}}\bmu_{i})=p^{j}+2(p^{i+j}-p^{j})-(2p^{i+j}-p^{j})=0$
\end{enumerate}
Every homogeneous class $x\in \bar E^{2}_{\ast,\star}$ in \eqref{equation:kssCtaualgebraicallyclosed} 
is a monomial in the generators $\blambda_{i+1}$ and $\gamma_{p^{j}}\bmu_{i}$.
The Chow degree $c(x)$ records the number of $\lambda_{i+1}$ classes in $x$, 
and the equality $0\leq c(x)\leq f_{x}$ follows from the definition.

\begin{lemma}
  \label{lem:MHH/taup-1}
The Tor spectral sequence \eqref{equation:kssCtaualgebraicallyclosed} for $\MHH_{\star}(\FF_{p})/\tau^{p-1}$ 
collapses at its $E^{2}$ page.
\end{lemma}
\begin{proof}
For $r\geq 2$ and $x\in E^{r}_{\ast,\star}$ we note the equality of weights $w_{x}=w_{d^{r}x}$.
If $x=\tau$, 
then $d^{r}\tau=0$ since \eqref{equation:kssCtaualgebraicallyclosed} is an $\FF_{p}[\tau]$-algebra spectral sequence.
If $x=\blambda_{i+1}$ or $x=\gamma_{p^{j}}\bmu_{i}$, 
the congruence $w_{d^{r}x}\equiv 0\bmod p-1$ shows the monomials in $d^{r}x$ are not $\tau$-divisible.
Hence,
$d^{r}x=0$, and we are done, or $c(d^{r}x)\geq 0$.
It remains to note that $c(d^{r}x)=c(x)-2r+1<0$.
\end{proof}

\begin{lemma}
There are no multiplicative extensions in the mod-$\tau^{p-1}$ Tor spectral sequence 
\eqref{equation:kssCtaualgebraicallyclosed}.
\end{lemma}
\begin{proof}{}
The Chow degree of $x$ equals $c(x)=2f_{x}+2w_{x}-(d_{x}+f_{x})$.
To find a hidden extension for $g=(\gamma_{p^{j}}\bmu_{i})^p=0$, 
we search among the $x$'s that satisfy
\begin{enumerate}
\item $d_{x}+f_{x}=d_g+f_g=2p^{i+j+1}$, 
\item $w_{x}=w_g=p(p^{i+j}-p^{j})$,
\item $0<f_{x}<f_g=p^{j+1}$.
\end{enumerate}
This rules out the existence of multiplicative extensions, 
since for the Chow degree, 
we have 
$$
c(x)
=
2f_{x}+2w_{x}-(d_{x}+f_{x})
=
2f_{x}+2p(p^{i+j}-p^{j})-2p^{j+i+1}
=
2(f_{x}-p^{j+1})
<0.
$$ 

Likewise, 
a hidden extension for $\blambda_{i+1}^2=0$ would be a class $x$ with $|x|=(1,4(p^{i+1}-1),2p^{i+1}-2)$; 
by inspection, no such class exists since all possible $x$ of filtration $1$ have weight $p^{j}-1$, $j\geq 0$.
\end{proof}

\begin{theorem}
\label{theorem:MHHMFCtau}
There is an isomorphism of graded commutative $\FF_{p,\tau}$-algebras
\begin{equation*}
\MHH_{\star}(\FF_p)/\tau^{p-1}
\cong
(\bigotimes_{i\geq 0}
\GAMMA_{\FF_{p}}(\bmu_{i})
\otimes 
\LAMBDA_{\FF_{p}}(\blambda_{i+1}))
\otimes
\FF_{p,\tau}.
\end{equation*}
The bidegrees of the generators are $|\bmu_{i}|=(2p^{i},p^{i}-1)$ and $|\blambda_{i+1}|=(2p^{i+1}-1,p^{i+1}-1)$.
\end{theorem}
 
\begin{remark}
  The reader may recognize the answer as $\MHH_{\star}(\FF_{p})/\tau^{p-1}\cong C\otimes \FF_{p,\tau}$ where $C=\TOR^{\AAA_{\star}^{\textrm{rig}}}_{\ast,\star}(\FF_{p},\FF_{p})$
 appeared in \Cref{sec:Dcomplex}.
\end{remark}

\subsection{Integral motivic Hochschild homology}
\label{subsection:imhh}

We now turn to the integral case of the Tor spectral sequence
\begin{equation}
E^{r}_{h,t,w}
\Rightarrow
\MHH_{h+t,w}(\M\FF_{p}).
\end{equation}

There is a natural comparison map $q\colon E^r_{\star;*}\rightarrow \bar E^r_{\star;*}$ to the mod-$\tau^{p-1}$ 
Tor spectral sequence analyzed in \S\ref{sec:modtau}.
Due to \Cref{theorem:MHHMFCtau} we have the following non-trivial mod-$\tau^{p-1}$ classes and their representatives 
in the bar complex:
\begin{enumerate}
\item $\blambda_{i+1}\in \MHH_{\star}(\FF_{p})/\tau^{p-1}$ is the class of the permanent cycle 
$[\bxi_{i+1}]\in \bar E^{1}_{1,2p^{i+1}-2;p^{i+1}-1}$, 
\item 
$\gamma_j\bmu_{i}\in \MHH_{\star}(\FF_{p})/\tau^{p-1}$ is the class of the permanent cycle
$[\btau_i|\dots|\btau_i]\in \bar E^{1}_{j,j(2p^{i}-1);j(p^{i}-1)}$.
\end{enumerate}
As before, to aid the bookkeeping we also set
$$\lambda_{i+1}=[\xi_{i+1}]\in E^{1}_{1,2p^{i+1}-2;p^{i+1}-1}$$ and
$$\gamma_j\mu_{i}=[\tau_i|\dots|\tau_i]\in E^{1}_{j,j(2p^{i}-1);j(p^{i}-1)},$$ even though the $\gamma_j\mu_i$s turn out to be permanent cycles for $j<p$ only.

As already noted, when $p$ is an odd prime 
  $E^2=\FF_p[\tau]\otimes \bigotimes_{i\geq 0}\Lambda(\lambda_{i+1})\otimes\Gamma(\mu_i)$.
\begin{lemma}
  Let  $p$ be a prime.
  \begin{itemize}
  \item For $0<r<p$ the {\'e}tale localization
$$L^r_{\text{{\'e}t}}\colon E^r\to E^r[\tau^{-1}]$$ is an injection.
\item For $1<r<p-1$, the differentials $d^r:E^r\to E^r$ are all zero. 
\item For all $p$
  $$d^{p-1}\gamma_{j+p}\mu_i\overset{\cdot}=\tau^{p-1}\lambda_{i+1}\gamma_j\mu_i$$
  for $i,j\geq 0$ and for odd $p$, this generates the  $d^{p-1}$-differential multiplicatively.
  \end{itemize}
\end{lemma}
\begin{proof}
  Since the dual Steenrod algebra has no $\tau$-torsion we have that $L^1_{\text{{\'e}t}}$ is an injection, and from the Tor-calculations we get that for odd primes $p$ also $L^2_{\text{{\'e}t}}$ is an injection.  Assume that for given $0<r<p$ $L^r_{\text{{\'e}t}}$ is injective.  For $1<r<p-1$ we have established that the differential on $E^r[\tau^{-1}]$ is trivial, and so the differential on $E^r$ is trivial too. Hence  $L^{r+1}_{\text{{\'e}t}}$ is injective, showing that (for odd primes $p$) $E^2=E^3=\dots=E^{p-1}$.

  Finally, since for all primes $p$ we now have $L^{p-1}_{\text{{\'e}t}}$ is an injection, the formula $d^{p-1}\gamma_{j+p}\mu_i\overset{\cdot}=\tau^{p-1}\lambda_{i+1}\gamma_j\mu_i$ follows from the same formula in $E^{p-1}[\tau^{-1}]$.
\end{proof}

The case for odd and even primes $p$ takes slightly different paths from here on.  The case $p=2$ is in many ways the simplest one but requires more care in that it turns out to be neither practical nor necessary to muddle through with the integral spectral sequence calculation: everything emanates from the torsion and $\tau$-inverted $\MHH$s together with minimal information about the integral $E^1$-page and an analysis of the Bockstein homology (called ``a Bockstein type complex'' in \Cref{sec:Dcomplex} since it also appears in the odd primary case in a slightly different guise) giving the answer --- with all multiplicative extensions --- without more ado.

\subsubsection{The even case}
\label{sec:even}
Let $p=2$. Since $\tau$ is a non-zero divisor in $\AAA_{\star}$, 
multiplication by $\tau$ gives the short exact sequence
$$\xymatrix{
0\ar[r] &
E^1_{f,d;w+1}
\ar[r]^{\tau}&
E^1_{f,d;w}
\ar[r]^q&
\bar E^1_{f,d;w}
0\ar[r] &
0.}
$$
We recall that the mod-$\tau$ spectral sequence collapses at $\bar E^2$ and has no multiplicative extensions: $\MHH_\star(\FF_2)/\tau\cong \bar E^2$.
Moving on to the abutment, 
the $\tau$-Bockstein on $\MHH_{\star}(\FF_{2})$ 
is the composite
\begin{equation}
\label{equation:Bocksteindefinition2}
\bar\partial
\colon
\MHH_{\ast+1,\ast}(\FF_{2})/\tau
\xrightarrow{\partial}
\MHH_{\ast,\ast+1}(\FF_{2})
\xrightarrow{q}
\MHH_{\ast,\ast+1}(\FF_{p})/\tau.
\end{equation}
Since \eqref{equation:Bocksteindefinition2} is a derivation we only need to know its value on the generators.
These are obtained from the integral $d^1$-differentials analyzed in \Cref{rem:firstdiff} as follows.
Since $\blambda_{i+1}$ is hit by the $d^1$-boundary $\lambda_{i+1}=[\xi_{i+1}]\in E^{1}_{1,2p^{i+1}-2;p^{i+1}-1}$ 
we get $\bar\partial\blambda_{i+1}=0$, 
and since $\gamma_{j+2}\bmu_i$ is hit by $\gamma_{j+2}\mu_i=[\tau_i|\dots|\tau_i]\in E^{1}_{j,j(2p^{i}-1);j(p^{i}-1)}$ 
and $d^1\gamma_{j+2}\mu_{i}=\tau\lambda_{i+1}\gamma_{j}\mu_{i}$ we deduce the following lemma.

\begin{lemma}
\label{lemma:nontrivialB}
The nontrivial $\tau$-Bocksteins on $\MHH_{\star}(\FF_{2})$ are generated by 
$$
\bar \partial \gamma_{j+2}\bmu_{i}=\blambda_{i+1}\gamma_{j}\bmu_{i}
$$
for all $i,j\geq 0$, i.e., $(\MHH_\star(\FF_2)/\tau,\bar\partial)=(C,D)$, where $(C,D)$ is the commutative differential graded algebra of \Cref{sec:Dcomplex}.
\end{lemma}
Combined with \Cref{lem:fracturebyfs}, 
and using that the $\tau$-free element $\mu_{i}\in \MHH_{\star}(\FF_{p})$ 
maps to $\bmu_{i}\in\MHH_{\star}(\FF_{p})/\tau^{p-1}$, 
we deduce the following result.

\begin{corollary}
\label{corollary:Bhomology}
The Bockstein homology of $\MHH_{\star}(\FF_2)/\tau$ is isomorphic to the graded commutative 
$\FF_2$-algebra $\bigoplus_{i\geq 0}\LAMBDA(\bmu_{i})$.
\end{corollary}

\Cref{corollary:Bhomology} lets us conclude that the $\tau$-torsion classes in $\MHH_{\star}(\FF_2)$ 
are not $\tau$-divisible.
The $\tau$-torsion in $\MHH_{\star}(\FF_2)$ agrees with the image of 
$\partial\colon \MHH_{\star}(\FF_2)/\tau\to \MHH_{\star}(\FF_2)$ and 
maps injectively via 
$q\colon \MHH_{\star}(\FF_2)\to \MHH_{\star}(\FF_2)/\tau$.
\vspace{0.1in} 

There is a naturally induced commutative diagram with exact rows
$$
\xymatrix{
0\ar[r]&(\tau\text{-torsion})\ar[r]\ar[d]^q_\cong & \MHH_{\star}(\FF_2)\ar[r]\ar[d]^q & 
\FF_2[\tau,\mu_0,\mu_1,\dots]/(\mu_{i}^2-\tau\mu_{i+1})\ar[r] & 0\\
0\ar[r] & \textrm{im}\bar\partial\ar[r] & 
\bigotimes_{i\geq 0}\GAMMA(\bmu_{i})\otimes \LAMBDA(\blambda_{i+1}). & &
}
$$
More elegantly, 
using \Cref{cor:HDandZD}, 
we have a pullback diagram of commutative $\FF_2[\tau]$-algebras
$$\xymatrix{
   \MHH_{\star}(\FF_2)\ar[r]\ar[d]&\FF_p[\tau,\mu_i]
   /\mu_i^2-\tau\mu_{i+1}\ar[d]\\
   \FF_2[\tau,\bmu_{i},\x_{S,f}]
   /\mathcal{I}
   \ar[r]&\FF_2[\tau,\bmu_i]
   /(\bmu_i^2,\tau)}
 $$
with indexation $i\in\NN$, $(S,f)\in K$ (see \Cref{def:JandK}), and 
$$\mathcal{I}=
(\tau,\bmu_{i}^2,\x_{S,f}\cdot\x_{T,g}-\sum_{t_{f+g}\not=u\in\supp(f+g)-S\cup T}\epsilon_{u}\cdot\x_{S\cup T\cup\{u\},f+g}).$$
Here $\mu_i$ maps to $\bmu_i$ and $\x_{S,f}$ maps to zero.
When we are done with the odd case, we'll see that by replacing $2$ with $p$, we have the general formula.

\subsubsection{The odd case}
\label{sec:odd}
Let $p$ be an odd prime.
The first task is using our knowledge of $d^{p-1}$ to calculate $E^p$.
Consider short exact sequence
$$\xymatrix{
0 
\ar[r]&
E^{p-1}_{f,d;w+p-1}
\ar[r]^-{\tau^{p-1}}&
E^{p-1}_{f,d;w}
\ar[r]^q&
\bar E^{p-1}_{f,d;w}
\ar[r]&
0}
$$
and the injection
$$L^{p-1}_{\text{{\'e}t}}\colon E^{p-1}_{f,d;w}\to E^{p-1}_{f,d;w}[\tau^{-1}].$$

\begin{defn}
  For $p\leq r$, let $P(r)$ be the conjunction of the propositions $P(r)_1$, $P(r)_2$, and $P(r)_3$ defined as follows:
  \begin{enumerate}
  \itemindent=2em%
  \item[$P(r)_1$:] $\xymatrix{
E^{r}_{f,d;w+p-1}
\ar[r]^-{\tau^{p-1}}&
E^{r}_{f,d;w}
\ar[r]^q&
\bar E^{r}_{f,d;w}}
$ is exact,
\item[$P(r)_2$:] in $E^{r}_{f,d;w}$ we have $\ker L^{r}_{\text{{\'e}t}} 
  =\ker \tau^{p-1}$, and
\item[$P(r)_3$:] for $p\leq j<r$ the $jth$ differential $d^j$ is trivial (so that $E^p=E^r$).
  \end{enumerate}
\end{defn}
To simplify notation, consider the $\FF_p$-algebra $C=\bigotimes_{i\geq 0}\GAMMA_{\FF_{p}}(\bmu_{i})
\otimes 
\LAMBDA_{\FF_{p}}(\blambda_{i+1})$ (with the above isomorphism $\bar E^r\cong C[\tau]/\tau^{p-1}$ for $r\geq 2$) and 
the derivation $D\colon C\to C$ generated by $D(\gamma_{j+p}\bmu_i)=\blambda_{i+1}\gamma_j\bmu_i$.
Let $B^D=\im D$, $Z^D=\ker D$ and $H^D=Z^D/B^D$.
\begin{lemma}
  \label{lem:P(p)}
  The proposition $P(p)$ is true.

Then $E^p$ is isomorphic to $Z^D[\tau]/\tau^{p-1}B^D[\tau]$ and under this isomorphism $E^p/\ker L^{p}_{\text{{\'e}t}}$ is isomorphic to $H^D[\tau]$.

Furthermore, the map $q\colon E^p\to\bar E^p\cong C[\tau]/\tau^{p-1}$ factors over $Z^D[\tau]/\tau^{p-1}\subseteq C[\tau]/\tau^{p-1}$ and the map
$$\ker\{\tau^{p-1}\colon E^p\to E^p\}\subseteq E^p\to \bar E^p\cong C[\tau]/\tau^{p-1}$$
is an injection factoring as an isomorphism $\ker\{\tau^{p-1}\colon E^p\to E^p\}\cong B^D[\tau]/\tau^{p-1}$ followed by the injection $B^D[\tau]/\tau^{p-1}\subseteq C[\tau]/\tau^{p-1}$.
Summing up, the resulting diagram of commutative $\FF_p[\tau]$-algebras
$$\xymatrix{E^p\ar[r]\ar[d]&H^D[\tau]\ar[d]\\
  Z^D[\tau]/\tau^{p-1}\ar[r]&H^D[\tau]/\tau^{p-1}}
  $$ is a pullback.
\end{lemma}
\begin{proof}
  For odd $p$, the first thing to notice is that $E^{p-1}$ is a free $\FF_p[\tau^{p-1}]$-module and that the differential factors $d^{p-1}=\tau^{p-1}{\bar\partial}$ where ${\bar\partial}$ (aka the Bockstein) is homogeneous with respect to the $\tau^{p-1}$-grading on $E^{p-1}$ and ${\bar\partial}^2=0$.
Let $Q$ be the degree zero part of $E^{p-1}$ (so that
  $E^{p-1}= Q[\tau^{p-1}]$ and
  $Q\subseteq E^{p-1}\to E^{p-1}/\tau^{p-1}\cong \bar E^{p-1}=\bar E^p$ is an isomorphism).
  If $Z^{\bar\partial}=\ker\{{\bar\partial}\colon Q\to Q\}$, then $\ker d^{p-1}=Z^{\bar\partial}[\tau^{p-1}]$, whereas if $B^{\bar\partial}=\im\{{\bar\partial}\colon Q\to Q\}$, then $\im \, d^{p-1}=\tau^{p-1}B^{\bar\partial}[\tau^{p-1}]$, 
  and if $H^{\bar\partial}=Z^{\bar\partial}/B^{\bar\partial}$, then (as an $\FF_p[\tau^{p-1}]$-module)
  $$E^p= Z^{\bar\partial}\oplus\tau^{p-1}H^{\bar\partial}[\tau^{p-1}],
  $$
  and $q\colon E^p\to\bar E^p$ may be identified with the composite
  $$Z^{\bar\partial}\oplus\tau^{p-1}H^{\bar\partial}[\tau^{p-1}]\to Z^{\bar\partial}\subseteq Q
  $$
  of the projection to the degree zero part followed by the inclusion.
  Hence
  $$\ker q=\tau^{p-1}H^{\bar\partial}[\tau^{p-1}]=\im\{\tau^{p-1}\colon E^p\to E^p\},$$ 
  $$\ker\{\tau^{p-1}\colon E^p\to E^p\}= B^{\bar\partial}=\ker L^{p}_{\text{{\'e}t}},
  $$
  and
  $$E^p/\ker \tau^{p-1} \cong H^{\bar\partial}[\tau^{p-1}]\subseteq H^{\bar\partial}[\tau^{\pm(p-1)}]\cong E^p[\tau^{-1}]
  $$
  Since $P(r)_3$ is vacuous in this case, we have proven $P(p)$.

  The formulation with the pullback follows when writing the above out as $\FF_p[\tau]$-algebras, 
  so that $\ker d^{p-1}=Z^D[\tau]$ and $\im \, d^{p-1}=\tau^{p-1}B^D[\tau]$ and remembering that $\bar E^{p-1}\cong C[\tau]/\tau^{p-1}$.
\end{proof}
\begin{lemma}
  \label{lem:Einftyanal1}
  For all $r\geq p$ the proposition $P(r)$ is true.  Hence,
  \begin{enumerate}
  \item $E^\infty=E^p$,
  \item the algebra map from the $\FF_p[\tau^{p-1}]$-free part to the $\tau^{p-1}$-localization
    $$L_{\text{{\'e}t}}\colon \left[\MHH_{\star}(\FF_{p})\right]/\ker\tau^{p-1} \to 
    \MHH_{\star}(\FF_{p})[\tau^{-1}]\cong \FF_{p}[\tau^{\pm 1},\mu_{i}]_{i\geq 0}/(\mu_{i}^{p}-\tau^{p-1}\mu_{i+1})
    $$ is injective so that $\left[\MHH_{\star}(\FF_{p})\right]/\ker\tau^{p-1}\cong \FF_p[\tau,\mu_i]_{i\geq 0}/(\mu_{i}^{p}-\tau^{p-1}\mu_{i+1})$,
  \item the  algebra map induced by $q\colon \MHH(\FF_{p})\to \MHH(\FF_{p})/\tau^{p-1}
    $
    $$q\colon \left[\MHH_{\star}(\FF_{p})\right]/\im\,\tau^{p-1}\to \MHH_{\star}(\FF_{p})/\tau^{p-1}\cong
(\bigotimes_{i\geq 0}
\GAMMA_{\FF_{p}}(\bmu_{i})
\otimes 
\LAMBDA_{\FF_{p}}(\blambda_{i+1}))
\otimes
\FF_{p,\tau}
    $$
    is injective, and
  \item the composite $\ker\tau^{p-1}\subseteq \MHH_{\star}(\FF_{p})\to 
  \left[\MHH_{\star}(\FF_{p})\right]/\im\,\tau^{p-1}$ is injective.
  \end{enumerate}
\end{lemma}
\begin{proof}
  By Lemma~\ref{lem:P(p)} we have $P(p)$ so we only need to show that $P(r)$ implies $P(r+1)$ for all $r\geq p$.
  Note that if $P(r)$ and $P(r+1)_3$ are true, then $P(r+1)$ is true.  Recall from Lemma~\ref{lemma:Einfinitytauinverted} and Lemma~\ref{lem:MHH/taup-1} that the $r$th differentials in both the localized and reduced $\Tor$-spectral sequences are trivial.
  
  Assume $P(r)$ and consider $x\in E^{r}_{f,d;w}$.  From the fact that $d^r\colon\bar E^r\to\bar E^r$ is trivial so that $0=d^rqx=qd^rx$ we get that there is a $y\in E^{r}_{f-r,d+r-1;w+p-1}$ so that $P(r)_1$ implies that $d^rx=\tau^{p-1}y$.
  Since $d^r\colon E^r[\tau^{-1}]\to\bar E^r[\tau^{-1}]$ is trivial we get that
  $0=d^rL^{r}_{\text{{\'e}t}}x=L^{r}_{\text{{\'e}t}}d^rx=L^{r}_{\text{{\'e}t}}\tau^{p-1}y=\tau^{p-1}L^{r}_{\text{{\'e}t}}y$ so that $0=L^{r}_{\text{{\'e}t}}y$ and $P(r)_2$ implies that $0=\tau^{p-1}y=d^rx$.

  The other points then follow directly, where in the last point we have used that 
  $\ker\tau^{p-1}=\ker L_{\text{{\'e}t}}$ gives that $\ker\tau^{p-1}\cap\im\,\tau^{p-1}=0$.
\end{proof}

Summing up in the language of Lemma~\ref{lem:P(p)}, we have achieved a pullback of commutative $\FF_p[\tau]$-algebras

\begin{equation}
\label{equation:firstEinfinitypullback}
\xymatrix{
E^\infty\ar[r]\ar[d]&H^D[\tau]\ar[d]\\
Z^D[\tau]/\tau^{p-1}\ar[r]&H^D[\tau]/\tau^{p-1}.
}
\end{equation}
Moreover, 
the pullback survives to the abutment in the sense that the maps out of $E^\infty$ are the associated graded versions of maps induced from maps of commutative ring spectra. 

We now set out analyzing $Z^D$ and $H^D$.

 \subsection{Multiplicative extensions}
 \label{sec:multex}

From Lemma~\ref{lem:Einftyanal1} we deduced the pullback \eqref{equation:firstEinfinitypullback} of commutative 
$\FF_p[\tau]$-algebras, 
which, 
given the information of Corollary~\ref{cor:HDandZD}, 
takes the form
$$
\xymatrix{
E^\infty\ar[r]\ar[d]&\FF_p[\tau,\mu_i]/\mu_i^p\ar[d]\\
\FF_p[\tau,\bmu_{i},\x_{S,f}]/\overline{\mathcal{I}}
\ar[r] & \FF_p[\tau,\bmu_i]/(\bmu_i^p,\tau^{p-1}) 
}
$$
with indexation $i\in\NN$, $(S,f)\in K$, 
and 
$$\overline{\mathcal{I}}=
(\tau^{p-1},\bmu_{i}^p,\x_{S,f}\cdot\x_{T,g}-\sum_{t_{f+g}\not=
u\in\supp(f+g)-S\cup T}\epsilon_{u}\cdot\x_{S\cup T\cup\{u\},f+g}).$$
Here
$\mu_i$ maps to $\bmu_i$ and $\x_{S,f}$ maps to zero.
Moreover, 
the pullback survives to the abutment in the sense that the maps out of $E^\infty$ are the associated graded versions 
of maps induced from maps of commutative ring spectra.
\vspace{0.1in}

In the abutment, 
we know that $\mu_i^p=\tau^{p-1}\mu_{i+1}$, but can there be further extensions?  
Since $\mu_i$ maps to $\bmu_i$, 
such an extension must be witnessed when passing from the associated graded 
$\FF_p[\tau,\bmu_{i},\x_{S,f}]/\overline{\mathcal{I}}$ to $\MHH_\star(\FF_p)/\tau^{p-1}$, 
but this we have seen in the mod $\tau^{p-1}$-calculation is not the case.  
In conclusion, 
we have shown the following result.
\begin{theorem}
  There is an isomorphism of graded commutative $\FF_p[\tau]$-algebras
  $$\MHH_*(\FF_p)\cong \FF_p[\tau,\mu_i,\x_{S,f}]_{i\in\NN, (S,f)\in K}/\mathcal{I}
  $$
  where the indexing set $K$ is given in Definition~\ref{def:JandK} and $\mathcal{I}$ is the ideal generated by
  \begin{itemize}
  \item $\mu_i^p-\tau^{p-1}\mu_{i+1}$,
  \item $\tau^{p-1}\x_{S,f}$, and
  \item $\x_{S,f}\cdot\x_{T,g}-\sum_u\epsilon_{u,S,T,f,g}\cdot\x_{S\cup T\cup\{u\},f+g}$ where the sum runs over all 
  elements $u\notin S\cup T$ so that $(f+g,S\cup T\cup\{u\})\in K$, 
  and the coefficient $\epsilon_{u,S,T,f,g}\in\FF_p$ is given in Definition~\ref{def:somenumbers}.
  \end{itemize}

\end{theorem}

\paragraph{Acknowledgments}
We thank Markus Spitzweck for his help with an argument in \Cref{sec:modtau}.
The authors acknowledge the support of the Centre for Advanced Study at the Norwegian Academy of Science and Letters 
in Oslo, Norway, which funded and hosted our research project ``Motivic Geometry" during the 2020/21 academic year.
This research was supported by grants from the RCN Frontier Research Group Project no. 250399 ``Motivic Hopf Equations" 
and no. 312472 ``Equations in Motivic Homotopy."

\begin{footnotesize}
\bibliographystyle{plain}
\bibliography{mhh}
\end{footnotesize}
\vspace{0.1in}

\begin{center}
Department of Mathematics, University of Bergen, Norway\\
email: dundas@math.uib.no
\end{center}
\begin{center}
Department of Mathematics, University of California, Los Angeles, USA\\
email: mikehill@math.ucla.edu
\end{center}
\begin{center}
Department of Mathematics,, Reed College, Portland, USA\\
email: ormsbyk@reed.edu
\end{center}
\begin{center}
Department of Mathematics F. Enriques, University of Milan, Italy \\ 
Department of Mathematics, University of Oslo, Norway \\
email: paul.oestvaer@unimi.it, paularne@math.uio.no
\end{center}
\end{document}